\documentclass[a4paper]{amsart}
\usepackage[T1]{fontenc}
\usepackage{amstext}
\usepackage{amsthm}
\usepackage{amssymb}
\usepackage{stackrel}
\usepackage[pdfusetitle,
 bookmarks=true,bookmarksnumbered=false,bookmarksopen=false,
 breaklinks=false,pdfborder={0 0 0},pdfborderstyle={},backref=false,colorlinks=false]
 {hyperref}

\makeatletter

\numberwithin{equation}{section}
\numberwithin{figure}{section}

\usepackage{tikz-cd}
\usepackage{mathtools}
\usepackage{adjustbox}
\usepackage{enumitem}
\tikzcdset{scale cd/.style={every label/.append style={scale=#1},
    cells={nodes={scale=#1}}}}
\usepackage{xcolor}
\usepackage{eucal}
\DeclareMathSymbol{:}{\mathpunct}{operators}{"3A}
\usepackage{mleftright}
\mleftright

\makeatother

\theoremstyle{plain}
\newtheorem{thm}{\protect\theoremname}[section]
\theoremstyle{definition}
\newtheorem{defn}[thm]{\protect\definitionname}
\newtheorem{example}[thm]{\protect\examplename}
\theoremstyle{remark}
\newtheorem{rem}[thm]{\protect\remarkname}
\theoremstyle{plain}
\newtheorem{cor}[thm]{\protect\corollaryname}
\theoremstyle{definition}
\newtheorem{notation}[thm]{\protect\notationname}
\theoremstyle{plain}
\newtheorem{lem}[thm]{\protect\lemmaname}
\newtheorem{prop}[thm]{\protect\propositionname}
\theoremstyle{definition}
\newtheorem{warning}[thm]{\protect\warningname}

\providecommand{\corollaryname}{Corollary}
\providecommand{\definitionname}{Definition}
\providecommand{\examplename}{Example}
\providecommand{\lemmaname}{Lemma}
\providecommand{\notationname}{Notation}
\providecommand{\propositionname}{Proposition}
\providecommand{\remarkname}{Remark}
\providecommand{\theoremname}{Theorem}
\providecommand{\warningname}{Warning}

\begin{document}
\global\long\def\sf#1{\mathsf{#1}}%

\global\long\def\scr#1{\mathscr{{#1}}}%

\global\long\def\cal#1{\mathcal{#1}}%

\global\long\def\bb#1{\mathbb{#1}}%

\global\long\def\bf#1{\mathbf{#1}}%

\global\long\def\frak#1{\mathfrak{#1}}%

\global\long\def\fr#1{\mathfrak{#1}}%

\global\long\def\u#1{\underline{#1}}%

\global\long\def\tild#1{\widetilde{#1}}%

\global\long\def\mrm#1{\mathrm{#1}}%

\global\long\def\pr#1{\left(#1\right)}%

\global\long\def\abs#1{\left|#1\right|}%

\global\long\def\inp#1{\left\langle #1\right\rangle }%

\global\long\def\br#1{\left\{  #1\right\}  }%

\global\long\def\norm#1{\left\Vert #1\right\Vert }%

\global\long\def\hat#1{\widehat{#1}}%

\global\long\def\opn#1{\operatorname{#1}}%

\global\long\def\bigmid{\,\middle|\,}%

\global\long\def\Top{\sf{Top}}%

\global\long\def\Set{\sf{Set}}%

\global\long\def\SS{\sf{sSet}}%

\global\long\def\Kan{\sf{Kan}}%

\global\long\def\Cat{\mathcal{C}\sf{at}}%

\global\long\def\Grpd{\mathcal{G}\sf{rpd}}%

\global\long\def\Res{\mathcal{R}\sf{es}}%

\global\long\def\imfld{\cal M\mathsf{fld}}%

\global\long\def\ids{\cal D\sf{isk}}%

\global\long\def\ich{\cal C\sf h}%

\global\long\def\SW{\mathcal{SW}}%

\global\long\def\SHC{\mathcal{SHC}}%

\global\long\def\Fib{\mathcal{F}\mathsf{ib}}%

\global\long\def\Bund{\mathcal{B}\mathsf{und}}%

\global\long\def\Fam{\cal F\sf{amOp}}%

\global\long\def\B{\sf B}%

\global\long\def\Spaces{\sf{Spaces}}%

\global\long\def\Mod{\sf{Mod}}%

\global\long\def\Nec{\sf{Nec}}%

\global\long\def\Fin{\sf{Fin}}%

\global\long\def\Ch{\sf{Ch}}%

\global\long\def\Ab{\sf{Ab}}%

\global\long\def\SA{\sf{sAb}}%

\global\long\def\P{\mathsf{POp}}%

\global\long\def\Op{\mathcal{O}\mathsf{p}}%

\global\long\def\Opg{\mathcal{O}\mathsf{p}^{\mathrm{gn}}_{\infty}}%

\global\long\def\Tup{\mathsf{Tup}}%

\global\long\def\H{\cal H}%

\global\long\def\Mfld{\cal M\mathsf{fld}}%

\global\long\def\D{\cal D\mathsf{isk}}%

\global\long\def\Acc{\mathcal{A}\mathsf{cc}}%

\global\long\def\Pr{\mathcal{P}\mathrm{\mathsf{r}}}%

\global\long\def\Del{\mathbf{\Delta}}%

\global\long\def\id{\operatorname{id}}%

\global\long\def\Aut{\operatorname{Aut}}%

\global\long\def\End{\operatorname{End}}%

\global\long\def\Hom{\operatorname{Hom}}%

\global\long\def\Ext{\operatorname{Ext}}%

\global\long\def\sk{\operatorname{sk}}%

\global\long\def\ihom{\underline{\operatorname{Hom}}}%

\global\long\def\N{\mathrm{N}}%

\global\long\def\-{\text{-}}%

\global\long\def\op{\mathrm{op}}%

\global\long\def\To{\Rightarrow}%

\global\long\def\rr{\rightrightarrows}%

\global\long\def\rl{\rightleftarrows}%

\global\long\def\mono{\rightarrowtail}%

\global\long\def\epi{\twoheadrightarrow}%

\global\long\def\comma{\downarrow}%

\global\long\def\ot{\leftarrow}%

\global\long\def\corr{\leftrightsquigarrow}%

\global\long\def\lim{\operatorname{lim}}%

\global\long\def\colim{\operatorname{colim}}%

\global\long\def\holim{\operatorname{holim}}%

\global\long\def\hocolim{\operatorname{hocolim}}%

\global\long\def\Ran{\operatorname{Ran}}%

\global\long\def\Lan{\operatorname{Lan}}%

\global\long\def\Sk{\operatorname{Sk}}%

\global\long\def\Sd{\operatorname{Sd}}%

\global\long\def\Ex{\operatorname{Ex}}%

\global\long\def\Cosk{\operatorname{Cosk}}%

\global\long\def\Sing{\operatorname{Sing}}%

\global\long\def\Sp{\operatorname{Sp}}%

\global\long\def\Spc{\operatorname{Spc}}%

\global\long\def\Fun{\operatorname{Fun}}%

\global\long\def\map{\operatorname{map}}%

\global\long\def\diag{\operatorname{diag}}%

\global\long\def\Gap{\operatorname{Gap}}%

\global\long\def\cc{\operatorname{cc}}%

\global\long\def\ob{\operatorname{ob}}%

\global\long\def\Map{\operatorname{Map}}%

\global\long\def\Rfib{\operatorname{RFib}}%

\global\long\def\Lfib{\operatorname{LFib}}%

\global\long\def\Tw{\operatorname{Tw}}%

\global\long\def\Equiv{\operatorname{Equiv}}%

\global\long\def\Arr{\operatorname{Arr}}%

\global\long\def\Cyl{\operatorname{Cyl}}%

\global\long\def\Path{\operatorname{Path}}%

\global\long\def\Alg{\operatorname{Alg}}%

\global\long\def\ho{\operatorname{ho}}%

\global\long\def\Comm{\operatorname{Comm}}%

\global\long\def\Triv{\operatorname{Triv}}%

\global\long\def\triv{\operatorname{triv}}%

\global\long\def\Env{\operatorname{Env}}%

\global\long\def\Act{\operatorname{Act}}%

\global\long\def\loc{\operatorname{loc}}%

\global\long\def\Assem{\operatorname{Assem}}%

\global\long\def\Nat{\operatorname{Nat}}%

\global\long\def\Conf{\operatorname{Conf}}%

\global\long\def\Rect{\operatorname{Rect}}%

\global\long\def\Emb{\operatorname{Emb}}%

\global\long\def\Homeo{\operatorname{Homeo}}%

\global\long\def\mor{\operatorname{mor}}%

\global\long\def\Germ{\operatorname{Germ}}%

\global\long\def\Post{\operatorname{Post}}%

\global\long\def\Sub{\operatorname{Sub}}%

\global\long\def\Shv{\operatorname{Shv}}%

\global\long\def\Cov{\operatorname{Cov}}%

\global\long\def\Disc{\operatorname{Disc}}%

\global\long\def\Tot{\operatorname{Tot}}%

\global\long\def\Ho{\operatorname{Ho}}%

\global\long\def\ABF{\mathbf{AB4}}%

\global\long\def\Top{\mathrm{Top}}%

\global\long\def\lax{\mathrm{lax}}%

\global\long\def\weq{\mathrm{weq}}%

\global\long\def\fib{\mathrm{fib}}%

\global\long\def\inert{\mathrm{inert}}%

\global\long\def\act{\mathrm{act}}%

\global\long\def\cof{\mathrm{cof}}%

\global\long\def\inj{\mathrm{inj}}%

\global\long\def\proj{\mathrm{proj}}%

\global\long\def\univ{\mathrm{univ}}%

\global\long\def\Ker{\opn{Ker}}%

\global\long\def\Coker{\opn{Coker}}%

\global\long\def\Im{\opn{Im}}%

\global\long\def\Coim{\opn{Coim}}%

\global\long\def\coker{\opn{coker}}%

\global\long\def\im{\opn{\mathrm{im}}}%

\global\long\def\coim{\opn{coim}}%

\global\long\def\gn{\mathrm{gn}}%

\global\long\def\Mon{\opn{Mon}}%

\global\long\def\Un{\opn{Un}}%

\global\long\def\St{\opn{St}}%

\global\long\def\cun{\widetilde{\opn{Un}}}%

\global\long\def\cst{\widetilde{\opn{St}}}%

\global\long\def\Sym{\operatorname{Sym}}%

\global\long\def\CA{\operatorname{CAlg}}%

\global\long\def\Ind{\operatorname{Ind}}%

\global\long\def\rd{\mathrm{rd}}%

\global\long\def\xmono#1#2{\stackrel[#2]{#1}{\rightarrowtail}}%

\global\long\def\xepi#1#2{\stackrel[#2]{#1}{\twoheadrightarrow}}%

\global\long\def\adj{\stackrel[\longleftarrow]{\longrightarrow}{\bot}}%

\global\long\def\btimes{\boxtimes}%

\global\long\def\ps#1#2{\prescript{}{#1}{#2}}%

\global\long\def\ups#1#2{\prescript{#1}{}{#2}}%

\global\long\def\hofib{\mathrm{hofib}}%

\global\long\def\cofib{\mathrm{cofib}}%

\global\long\def\Vee{\bigvee}%

\global\long\def\w{\wedge}%

\global\long\def\t{\otimes}%

\global\long\def\bp{\boxplus}%

\global\long\def\rcone{\triangleright}%

\global\long\def\lcone{\triangleleft}%

\global\long\def\S{\mathsection}%

\global\long\def\p{\prime}%

\global\long\def\pp{\prime\prime}%

\global\long\def\W{\overline{W}}%

\global\long\def\rel{\mathrm{rel}}%

\title{Homotopy Limits and Homotopy Colimits of Chain Complexes}
\author{Kensuke Arakawa}
\email{arakawa.kensuke.22c@st.kyoto-u.ac.jp}
\address{Department of Mathematics, Kyoto University, Kyoto, 606-8502, Japan}
\subjclass[2020]{18G35, 55U15, 57T30}
\keywords{Homotopy colimits, bar construction, chain complexes}
\begin{abstract}
We give a formula for homotopy limits and homotopy colimits of diagrams
of chain complexes using the cobar and bar constructions, also known
as the Bousfield--Kan formula. Along the way, we show that the Bousfield--Kan
formula computes homotopy colimits in any framed model category.

\tableofcontents{}
\end{abstract}

\maketitle

\section*{Introduction}

Ordinary colimits (and dually, limits) do not get along well with
homotopical considerations. So when we think about objects up to homotopy,
we instead need to use homotopy colimits, which satisfy the homotopical
version of the universal properties of limits and colimits. A standard
approach to compute homotopy colimits is to deform a diagram to a
homotopically better behaved (e.g., projectively cofibrant) diagram,
and then taking the colimit of it. However, very often, this does
not give us anything concrete, because the deformation becomes quite
complicated as soon as the diagram gets mildly complex. 

For simplicial model categories, there is an alternative approach
to computing homotopy colimits, due to Bousfield and Kan \cite{BK72}.
If $\cal C$ is a simplicial model category, then the homotopy colimit
of a pointwise cofibrant diagram $F:\cal I\to\cal C$ is modeled by
the \textbf{bar construction} $B\pr{\ast,\cal I,F}$ of $F$, which
is the geometric realization of the simplicial object $B_{\bullet}\pr{\ast,\cal I,F}$
defined by
\[
B_{n}\pr{\ast,\cal I,F}=\coprod_{f:[n]\to\cal I}F\pr{f\pr 0}.
\]
(See \cite[Chapter 5]{cathtpy} for a textbook account.) Compared
to the colimit of the mystical projective cofibrant replacement, the
Bousfield--Kan formula gives us a very concrete model of homotopy
colimits.\footnote{Some coend calculus shows that $B\pr{\ast,\cal I,F}$ is isomorphic
to $N\pr{\cal I_{-/}}\otimes_{\cal I}F$, where $N$ denotes the nerve
functor and $\otimes_{\cal I}$ denotes the functor tensor product
\cite[Theorem 6.6.1]{cathtpy}. Historically, Bousfield and Kan introduced
homotopy colimits by using this functor tensor product.}

We can try to blindly apply this formula to diagrams of chain complexes:
If $\cal A$ is a cocomplete abelian category and $F:\cal I\to\Ch\pr{\cal A}$
is a small diagram, we can form the simplicial object $B_{\bullet}\pr{\ast,\cal I,F}$
as above. We can then assemble this simplicial object into a single
chain complex $B\pr{\ast,\cal I,F}$ by taking its ``geometric realization''
\[
B\pr{\ast,\cal I,F}=\int^{[n]\in\Del}N_{\ast}\pr{\Delta^{n}}\otimes B_{n}\pr{\ast,\cal I,F},
\]
where $N_{\ast}$ denotes the normalized chain complex (Definition
\ref{def:normalized}) of simplicial sets. (This geometric realization
is given by the direct sum totalization of the double complex associated
to $B_{\bullet}\pr{\ast,\cal I,F}$; see Proposition \ref{prop:realizatoin_totalization}.)
However, since $\Ch\pr{\cal A}$ often cannot be made into a simplicial
model category,\footnote{We can make $\Ch\pr{\cal A}$ into a simplicial category using the
Dold--Kan correspondence, but the resulting simplicial category is
almost never tensored over simplicial sets \cite[Warning 1.3.5.4]{HA}.
In particular, we often cannot make it into a simplicial model category.} it is not clear whether $B\pr{\ast,\cal I,F}$ models homotopy colimits.

Our main results, of which there are two, give precise conditions
under which the Bousfield--Kan formula computes homotopy colimits
of chain complexes. To state the first result, recall that the category
$\Ch\pr{\cal A}$ frequently admits a model structure whose weak equivalences
are quasi-isomorphisms. We then prove the following:
\begin{thm}
[Theorem \ref{thm:main1_precise}]\label{thm:main1}Let $\cal A$
be bicomplete abelian category, and suppose $\Ch\pr{\cal A}$ is equipped
with a model structure. Under a mild assumption on the model structure,
the homotopy colimit of a small, pointwise cofibrant diagram $F:\cal I\to\Ch\pr{\cal A}$
is modeled by $B\pr{\ast,\cal I,F}$.
\end{thm}

The homotopy theory of chain complexes is among the cleanest, so one
wonders if model structures are necessary at all for the current discussion.
Our second main result addresses this point:
\begin{thm}
[Theorem \ref{thm:main2_precise}]\label{thm:main2}Let $\cal A$
be an abelian category, and let $\kappa$ be a regular cardinal. Suppose
that $\cal A$ has $\kappa$-small coproducts. The following conditions
are equivalent:
\begin{enumerate}
\item Monomorphisms in $\cal A$ are stable under $\kappa$-small coproducts.
\item For every $\kappa$-small diagram $F:\cal I\to\Ch\pr{\cal A}$, the
bar construction $B\pr{\ast,\cal I,F}$ models the homotopy colimit
of $F$ (with respect to quasi-isomorphisms).
\end{enumerate}
\end{thm}

For example, by taking $\kappa=\omega$, we deduce that homotopy colimits
of chain complexes indexed by finite categories (i.e., categories
with only finitely many morphisms) can always be modeled by the Bousfield--Kan
formula.

Of course, there are dual versions of these theorems, relating homotopy
limits with cobar constructions, and they are included in Theorems
\ref{thm:main1_precise} and \ref{thm:main2_precise}.

Here is an outline of this paper. This paper has three sections:
\begin{enumerate}
\item In Section \ref{sec:BK}, we recall the theory of deformations, and
show that the Bousfield--Kan formula and its variations fit into
the framework of deformations (Theorem \ref{thm:BK} and Corollary
\ref{cor:BK}). This will tie the loose end left by \cite{Hirschhorn},
where homotopy colimits are \textit{defined} via the Bousfield--Kan
formula and never compared to one produced by deformation. 
\item Section \ref{sec:geom} studies geometric realization of simplicial
chain complexes. We show that, under the Dold--Kan correspondence,
geometric realization corresponds to totalization of double complexes.
We then establish a few properties of geometric realization, which
we need in Section \ref{sec:main}.
\item In Section \ref{sec:main}, we give proofs of Theorems \ref{thm:main1}
and \ref{thm:main2}. We also explain that many model structures on
chain complexes satisfy the hypothesis of Theorem \ref{thm:main1}.
\end{enumerate}

\subsection*{Related Work}

Given its classical nature, the topic of this paper has been partially
covered in many different papers. For example, homotopy limits of
positive chain complexes of modules over a ring is considered in \cite[Section 4]{Porter77}.
A version of Theorem \ref{thm:main1} for homotopy colimits of chain
complexes of vector spaces or modules over a ring appears in \cite[Section 2.2]{MR4163860}.
Also, the implication (1)$\implies$(2) of Theorem \ref{thm:main2}
for homotopy colimits of positive chain complexes in AB4 abelian categories
appears in \cite{Rodriguez14}. All of these papers use different
techniques from ours, which are of independent interest. The references
listed here are likely not exhaustive, but we are not aware of a comprehensive
treatment of the subject of this paper.

\subsection*{Notation and Convention}
\begin{itemize}
\item Model categories are assumed to be bicomplete and have functorial
factorizations.
\item We let $\Del$ denote the category whose objects are the posets $[n]=\{0,\dots,n\}$,
where $n\geq0$, and whose morphisms are the poset maps. We let $\Del_{\inj}\subset\Del$
denote the subcategory spanned by the injective poset maps. For each
$n\geq0$, we will write $\Del_{\leq n}\subset\Del$ and $\Del_{\inj,\leq n}\subset\Del_{\inj}$
for the full subcategory spanned by the objects $[0],\dots,[n]$. 
\item Let $\kappa$ be a regular cardinal. A category is said to be \textbf{$\kappa$-small}
if its set of morphisms has cardinality less than $\kappa$.
\item Let $\kappa$ be a regular cardinal. 
\begin{itemize}
\item An \textbf{$\ABF_{\kappa}$ abelian category} is an abelian category
with $\kappa$-small coproducts (i.e., coproducts indexed by sets
of cardinality less than $\kappa$), such that monomorphisms are stable
under $\kappa$-small coproducts. 
\item An \textbf{$\ABF$ abelian category} is an abelian category which
is $\ABF_{\lambda}$ for any regular cardinal $\lambda$.
\item An $\ABF^{*}_{\kappa}$ abelian category is an abelian category whose
opposite is $\ABF_{\kappa}$.
\item An $\ABF^{*}$ abelian category is an abelian category whose opposite
is $\ABF$.
\end{itemize}
\item We write $\omega=\aleph_{0}$ for the first infinite cardinal, and
write $\Omega=\aleph_{1}$ for the first uncountable cardinal.
\item Let $\cal C$ be a model category, and let $\cal I$ be a category.
We say that a functor $F:\cal I\to\cal C$ is \textbf{pointwise cofibrant}
if $F$ carries each object to a cofibrant object. We define pointwise
fibrant diagrams, pointwise weak equivalences (also called natural
weak equivalences), etc, in a similar manner.
\end{itemize}

\section{\label{sec:BK}Bousfield--Kan Formula in (Weakly) Framed Model Categories}

Let $\cal C$ be a model category. A \textit{weak cosimplicial framing}
on $\cal C$, roughly speaking, is a functorial choice of cosimplicial
resolutions of cofibrant objects in $\cal C$ (Definition \ref{def:weak_frame}).
If $\cal C$ is a weakly cosimplicially framed model category, we
may associate to each pointwise cofibrant simplicial object its \textit{geometric
realization} (Definition \ref{def:realization}). Using this, we can
formally mimic the Bousfield--Kan formula. The goal of this section
is to show that a minor variation of this formula computes homotopy
colimits.

We start by recalling the definition of derived functors, which we
use to construct homotopy colimits (Subsection \ref{subsec:der}).
In Subsection \ref{subsec:BK}, we show that the ``fat'' version
of the Bousfield--Kan formula can be used to derive colimits, and
give a condition under which we can reduce it to the ordinary Bousfield--Kan
formula (Theorem \ref{thm:BK} and Corollary \ref{cor:BK}).

\subsection{\label{subsec:der}Homotopy Colimits and Derived Functors}

In this subsection, we recall the definition of homotopy colimits
and explain their relation to derived functors. The contents of this
subsection is mostly a retelling of \cite[Chapter 2]{cathtpy}. 
\begin{defn}
\label{def:homotopical}A \textbf{relative category} is a category
$\cal C$ equipped with a subcategory whose morphisms are called \textbf{weak
equivalences} and which contains all objects of $\cal C$. If $\cal D$
is another relative category, a functor $\cal C\to\cal D$ is said
to be\textbf{ relative} if it preserves weak equivalences.

If $\cal C$ is a relative category, the \textbf{localization} of
$\cal C$ at weak equivalences is a functor $\cal C\to\Ho\pr{\cal C}$
which is characterized (up to equivalence) by the following universal
property: For every category $\cal E$, the functor
\[
\Fun\pr{\Ho\pr{\cal C},\cal E}\to\Fun\pr{\cal C,\cal E}
\]
is fully faithful, and its essential image consists of those functors
$\cal C\to\cal E$ that carry weak equivalences to isomorphisms. We
refer to $\Ho\pr{\cal C}$ as the \textbf{homotopy category} of $\cal C$.
We generally do not notationally distinguish between objects in $\cal C$
and their images in the homotopy category.
\end{defn}

\begin{example}
Every model category can be regarded as a relative category. If $\cal C$
is a relative category and $\cal I$ is a category, we will regard
$\cal C^{\cal I}$ as a relative category by declaring that its weak
equivalences are the natural weak equivalences, i.e., natural transformations
whose components are weak equivalences.
\end{example}

\begin{defn}
\label{def:hocolim}Let $\cal C$ be a relative category, and let
$\cal I$ be another category. The diagonal functor $\delta:\cal C\to\cal C^{\cal I}$
is a relative functor, so it induces a functor $\Ho\pr{\delta}:\Ho\pr{\cal C}\to\Ho\pr{\cal C^{\cal I}}$.
The \textbf{homotopy colimit }functor $\hocolim_{\cal I}:\Ho\pr{\cal C^{\cal I}}\to\Ho\pr{\cal C}$,
if it exists, is defined as the left adjoint of $\Ho\pr{\delta}$. 
\end{defn}

\begin{rem}
Definition \ref{def:hocolim} says nothing about the existence of
homotopy colimit functors. We will see in Theorem \ref{thm:BK} that
they \textit{always} exist for small diagrams in model categories.
(This can also be proved by resorting to $\infty$-categorical calculus
of fractions \cite[Remark 7.9.10]{CisinskiHCHA}.) 
\end{rem}

We will see that homotopy colimit functors arise as ``best homotopical
approximations'' to ordinary colimit functors. To make this more
precise, we need the notion of derived functors.
\begin{defn}
\cite[Definitions 2.1.17, 2.1.19]{cathtpy}\label{def:der} Let $\cal C$
and $\cal D$ be relative categories, and let $\gamma_{\cal C}:\cal C\to\Ho\pr{\cal C}$
and $\gamma_{\cal D}:\cal D\to\Ho\pr{\cal D}$ denote the localizations
at weak equivalences. 
\begin{enumerate}
\item A \textbf{total left derived functor} of $F$ is a functor $\bf LF:\Ho\pr{\cal C}\to\Ho\pr{\cal D}$
equipped with a natural transformation depicted as
\[\begin{tikzcd}
	{\mathcal{C}} & {\mathcal{D}} \\
	{\operatorname{Ho}(\mathcal{C})} & {\operatorname{Ho}(\mathcal{D}),}
	\arrow[""{name=0, anchor=center, inner sep=0}, "F", from=1-1, to=1-2]
	\arrow["{\gamma_\mathcal{C}}"', from=1-1, to=2-1]
	\arrow["{\gamma_\mathcal{D}}", from=1-2, to=2-2]
	\arrow[""{name=1, anchor=center, inner sep=0}, "{\mathbf{L}F}"', from=2-1, to=2-2]
	\arrow[shorten <=4pt, shorten >=4pt, Rightarrow, from=1, to=0]
\end{tikzcd}\]which exhibits $\bf LF$ as a \textit{right} Kan extension of $\gamma_{\cal D}\circ F$
along $\gamma_{\cal C}$. If further this is an absolute right Kan
extension (i.e., for any functor $G:\Ho\pr{\cal D}\to\cal E$, the
natural transformation $G\circ\bf LF\circ\gamma_{\cal C}\to G\circ\gamma_{\cal D}\circ F$
remains to exhibit $G\circ\bf LF$ as a right Kan extension), we say
that the total left derived functor is \textbf{absolute}.
\item A \textbf{left derived functor} of $F$ is a relative functor $\bb LF:\pr{\cal C,\cal W_{\cal C}}\to\pr{\cal D,\cal W_{\cal D}}$
equipped with a natural transformation $\lambda:\bb LF\To F$ with
the following property: Let $\bf LF:\Ho\pr{\cal C}\to\Ho\pr{\cal D}$
be any functor that admits a natural isomorphism $\bf LF\circ\gamma_{\cal C}\cong\gamma_{\cal D}\circ\bb LF$.
Then the composite
\[
\bf LF\circ\gamma_{\cal C}\cong\gamma_{\cal D}\circ\bb LF\stackrel{\lambda}{\To}\gamma_{\cal D}\circ F
\]
exhibits $\bf LF$ as a total left derived functor of $F$. If further $\bf LF$ is an absolute total left
derived functor, then we say that the left derived functor $\bb LF$
is \textbf{absolute}. 
\end{enumerate}
\textbf{Total right derived functors} and \textbf{right derived functors}
are defined dually.
\end{defn}

We now introduce a standard technique to construct derived functors.
\begin{defn}
\cite[Definition 2.2.4, Lemma 5.1.6]{cathtpy}\label{def:deform}
Let $\cal C$ and $\cal D$ be relative categories, and let $F:\cal C\to\cal D$
be a functor. A \textbf{left deformation }for $F$ is a natural transformation
$q:Q\To\id_{\cal C}$ of endofunctors of $\cal C$, satisfying the
following pair of conditions:
\begin{enumerate}
\item The natural transformation $q$ is a natural weak equivalence.
\item The functor $F\circ Q:{\cal C}\to\cal D$ is relative.
\end{enumerate}
If $F$ admits a left deformation, we say that $F$ is \textbf{left
deformable}.\textbf{ }We define \textbf{right deformations} similarly.
\end{defn}

\begin{thm}
\cite[Theorem 2.2.13]{cathtpy}\label{thm:2.2.8} Let $\cal C$ and
$\cal D$ be relative categories, and let $F:\cal C\to\cal D$ be
a functor. If $F$ admits a left deformation $q:Q\To\id_{\cal C}$,
then the pair $\pr{F\circ Q,Fq}$ is an absolute left derived functor
of $F$. 
\end{thm}

When absolute derived functors exist for a pair of adjoint functors,
they again form an adjoint pair:
\begin{thm}
\cite[Theorem 2.2.11]{cathtpy}, \cite{Mal07}\label{thm:2.2.11} Let
$\cal C$ and $\cal D$ be relative categories, and let $F:\cal C\adj\cal D:G$
be an adjoint pair of functors of underlying categories. If $F$ admits
an absolute total left derived functor and $G$ admits an absolute
total right derived functor, then the total derived functors $\mathbf{L}F:\Ho\pr{\cal C}\rl\Ho\pr{\cal D}:\mathbf{R}G$
are part of an adjunction characterized by the property that, for
every $C\in\cal C$ and $D\in\cal D$, the diagram
\[\begin{tikzcd}
	{\operatorname{Hom}_{\mathcal{D}}(F(C),D)} & {\operatorname{Hom}_{\mathcal{C}}(C,G(D))} \\
	{\operatorname{Hom}_{\operatorname{Ho}(\mathcal{D})}(\mathbf{L}F(C),D)} & {\operatorname{Hom}_{\operatorname{Ho}(\mathcal{C})}(C,\mathbf{R}G(D))}
	\arrow["\cong", from=1-1, to=1-2]
	\arrow[from=1-1, to=2-1]
	\arrow[from=1-2, to=2-2]
	\arrow["\cong"', from=2-1, to=2-2]
\end{tikzcd}\]commutes.
\end{thm}

If $\cal C$ is a model category and $\cal I$ is a small category,
the diagonal functor $\cal C\to\cal C^{\cal I}$ is a relative functor,
so it admits an absolute total right derived functor. Therefore, if
the colimit functor $\colim_{\cal I}:\cal C^{\cal I}\to\cal C$ admits
a left deformation, we can use Theorem \ref{thm:2.2.11} to construct
the homotopy colimit functor. This is how we construct the homotopy
colimit functor in the next subsection.

\subsection{\label{subsec:BK}Bousfield--Kan Formula in Framed Model Categories}

In this section, we define weakly simplicially framed model categories,
and show that variations of Bousfield--Kan formula model homotopy
colimits in these categories (Theorem \ref{thm:BK}, Corollary \ref{cor:BK}).

We start with a few definitions.
\begin{defn}
\label{def:resolution}Let $\cal C$ be a model category, let $\cal I$
be a category, and let $F:\cal I\to\cal C$ be a functor. A \textbf{cosimplicial
resolution} of $F$ is a functor $\bf F:\Del\times\cal I\to\cal C$
equipped with a natural weak equivalence $\alpha:\bf F\xrightarrow{\simeq}F\circ\opn{pr}$,
where $\opn{pr}:\Del\times\cal I\to\cal I$ denotes the projection,
such that for each $i\in\cal I$, the cosimplicial object $\bf F\pr i$
is Reedy cofibrant. We let ${\rm csRes}\pr F\subset\Fun\pr{\Del\times\cal I,\cal C}_{/F\circ\opn{pr}}$
denote the full subcategory spanned by the cosimplicial resolution
of $F$, which is a (possibly large) weakly contractible category
\cite[Theorem 14.5.4]{Hirschhorn}.\footnote{The proof in loc. cit. contains an error, which is corrected in \url{https://math.mit.edu/~psh/MCATL-errata-2018-08-01.pdf}.}
We define \textbf{simplicial resolutions} of $F$ dually.
\end{defn}

\begin{defn}
\label{def:weak_frame}Let $\cal C$ be a model category. A \textbf{weak
cosimplicial framing} on $\cal C$ is a cosimplicial resolution of
the inclusion $\cal C^{c}\hookrightarrow\cal C$, where $\cal C^{c}\subset\cal C$
denotes the full subcategory of cofibrant objects. The corresponding
bifunctor $\Del\times\cal C^{c}\to\cal C$ determines, via left Kan
extension along the Yoneda embedding, a bifunctor $\SS\times\cal C^{c}\to\cal C$.
We typically denote this bifunctor by $\otimes$ and abuse language
by saying that $\otimes$ is a weak cosimplicial framing. (Note that
for each $C\in\cal C^{c}$, the functor $-\otimes C:\SS\to\cal C$
is left Quillen \cite[Proposition 16.5.6]{Hirschhorn}.) If the bifunctor
$\otimes$ can be extended to a left Quillen bifunctor $\SS\times\cal C\to\cal C$,
we say that the weak cosimplicial framing is \textbf{excellent}. A
model category equipped with a weak cosimplicial framing is called
a \textbf{weakly cosimplicially framed model category}. 

Dually, a \textbf{weak simplicial framing }on $\cal C$ is a weak
cosimplicial framing on $\cal C^{\op}$. The corresponding bifunctor
will often be denoted by $\SS^{\op}\times\cal C\to\cal C$, $\pr{K,C}\mapsto C^{K}$.

A \textbf{weakly framed model category} is a model category equipped
with a weak cosimplicial framing and a weak simplicial framing.
\end{defn}

\begin{rem}
Every framed model category in the sense of \cite[Definition 16.6.21]{Hirschhorn}
has a canonical weak framing, and this is why we use the adjective
``weak''. Since every model category has a framing \cite[Theorem 16.6.9]{Hirschhorn},
it follows in particular that every model category admits a weak framing.
\end{rem}

\begin{example}
\label{exa:enr->wst}Every simplicial model category admits an excellent
simplicial framing, given by tensors by simplicial sets. More generally,
we can endow every enriched model category with an excellent weak
simplicial framing: Recall that a \textbf{symmetric monoidal model
category} is a symmetric monoidal category $\pr{\cal V,\otimes,\bf 1}$
equipped with a model structure satisfying the following pair of axioms:
\begin{enumerate}
\item The tensor bifunctor $\otimes:\cal V\times\cal V\to\cal V$ is a left
Quillen bifunctor.
\item Let $q:\widetilde{\bf 1}\to\bf 1$ be a weak equivalence, where $\widetilde{\bf 1}$
is cofibrant. For every cofibrant object $X\in\cal V$, the map $q\otimes X$
is a weak equivalence.
\end{enumerate}
A \textbf{model $\cal V$-category} is a $\cal V$-enriched category
$\cal M$ with a model structure on its underlying category, such
that the tensor bifunctor $\otimes:\cal V\times\cal M\to\cal M$ is
a left Quillen bifunctor.

Let $\cal V$ be a symmetric monoidal model category, and let $\cal M$
be a model $\cal V$-category. A choice of a cosimplicial resolution
of the unit object $\bf 1\in\cal V$ determines a left Quillen functor
$F:\SS\to\cal V$ \cite[Proposition 16.5.6]{Hirschhorn}. The composite
\[
\SS\times\cal M\xrightarrow{F\times\id}\cal V\times\cal M\xrightarrow{\otimes}\cal M
\]
gives rise to an excellent weak simplicial framing on $\cal M$.
\end{example}

\begin{defn}
\label{def:realization}Let $\cal C$ be a category equipped with
bifunctors $\otimes:\SS\times\cal C\to\cal C$ and $\pr -^{-}:\SS^{\op}\times\cal C\to\cal C$.
\begin{itemize}
\item The \textbf{geometric realization} of a simplicial object $X\in\cal C^{\Del^{\op}}$
is defined by the coend (provided that it exists)
\[
\abs X=\int^{[n]\in\Del}\Delta^{n}\otimes X_{n}.
\]
\item The \textbf{fat geometric realization} of a semi-simplicial object
$X\in\cal C^{\Del^{\op}_{\inj}}$ is defined by the coend
\[
\norm X=\int^{[n]\in\Del_{\inj}}\Delta^{n}\otimes X_{n}
\]
\item The \textbf{totalization} of a cosimplicial object $Y\in\cal C^{\Del}$
is defined by the end
\[
\Tot\pr Y=\int_{[n]\in\Del}\pr{Y^{n}}^{\Delta^{n}}.
\]
\item The \textbf{fat totalization} of a semi-simplicial object $Y\in\cal C^{\Del^{\op}_{\inj}}$
is defined by the end
\[
\Tot^{{\rm fat}}\pr Y=\int_{[n]\in\Del_{\inj}}\pr{Y^{n}}^{\Delta^{n}}.
\]
\end{itemize}
\end{defn}

\begin{defn}
\label{def:bar_chain}Let $\cal C$ and $\cal I$ be categories, and
let $F:\cal I\to\cal C$ be a diagram.
\begin{enumerate}
\item Let $W:\cal I^{\op}\to\Set$ be a diagram. The\textbf{ simplicial
bar construction} $B_{\bullet}\pr{W,\cal I,F}$ is the simplicial
object in $\cal C$ defined by
\[
B_{n}\pr{W,\cal I,F}=\coprod_{i_{0}\to\cdots\to i_{n}}W\pr{i_{n}}\cdot F\pr{i_{0}},
\]
where the coproduct is indexed by the functors $[n]\to\cal I$, and
for a set $S$ and $C\in\cal C$, we wrote $S\cdot C=\coprod_{s\in S}C$.
Here we tacitly assume that the relevant coproducts exist. 

If $\cal C$ is equipped with a bifunctor $\SS\times\cal C\to\cal C$,
the geometric realization of $B_{\bullet}\pr{W,\cal I,F}$ is called
the \textbf{bar construction} and is denoted by $B\pr{W,\cal I,F}$.
The fat geometric realization of $B_{\bullet}\pr{W,\cal I,F}$ is
called the \textbf{fat bar construction} and is denoted by $B^{{\rm fat}}\pr{W,\cal I,F}$.
\item Let $W:\cal I\to\Set$ be a diagram. The \textbf{cosimplicial cobar
construction} of $F$ and $W$ is the cosimplicial object $C^{\bullet}\pr{W,\cal I,F}\in\cal C^{\Del}$
whose $n$th term is given by
\[
C^{n}\pr{W,\cal I,F}=\prod_{i_{0}\to\cdots\to i_{n}}F\pr{i_{n}}^{W\pr{i_{0}}}.
\]
If $\cal C$ is equipped with a bifuncor $\SS^{\op}\times\cal C\to\cal C$,
the totalization of $C^{\bullet}\pr{W,\cal I,F}$ is called the \textbf{cobar
construction} and is denoted by $C\pr{W,\cal I,F}$. The fat totalization
of $C^{\bullet}\pr{W,\cal I,F}$ is called the \textbf{fat cobar construction}
and is denoted by $C^{{\rm fat}}\pr{W,\cal I,F}$. 
\end{enumerate}
\end{defn}

\begin{rem}
\label{rem:not_dual}We can think of the cosimplicial cobar construction
as a ``dual'' of the simplicial bar construction in the following
sense. Suppose we are given functors $W:\cal I\to\Set$ and $F:\cal I\to\cal C$.
From $W$ and $F$, we can construct two cosimplicial objects $\Del\to\cal C$:
\begin{enumerate}
\item We can form the cosimplicial cobar construction $C^{\bullet}\pr{W,\cal I,F}$.
\item Let $F^{\op}$ denote the functor $F$, thought of as a functor $\cal I^{\op}\to\cal C^{\op}$.
We can then form the simplicial bar construction $B_{\bullet}\pr{W,\cal I^{\op},F^{\op}}:\Del^{\op}\to\cal C^{\op}$.
We can regard this as a cosimplicial object $\Del\to\cal C$, which
we denote by $B_{\bullet}\pr{W,\cal I^{\op},F^{\op}}^{\op}$.
\end{enumerate}
These two cosimplicial objects are related by the opposition functor
$\pr -^{\op}:\Del\to\Del$, which carries a poset map $f:[n]\to[m]$
to the poset map $f^{\op}:[n]\to[m]$ defined by $f^{\op}\pr{n-i}=m-f\pr i$.
More precisely, $C^{\bullet}\pr{W,\cal I,F}$ is equal to the composite
\[
\Del\xrightarrow{\pr -^{\op}}\Del\xrightarrow{B_{\bullet}\pr{W,\cal I^{\op},F^{\op}}^{\op}}\cal C.
\]

Because of this, we will frequently say that results on cobar constructions
are ``dual'' to those of bar constructions, trusting that the readers
can make necessary changes if necessary.
\end{rem}

We can now state the main result of this subsection.
\begin{thm}
\label{thm:BK}Let $\cal C$ be a weakly framed model category with
a cofibrant replacement $Q\to\id_{\cal C}$ and a fibrant replacement
$\id_{\cal C}\to R$. Let $\cal I$ be a small category.
\begin{enumerate}
\item The composite natural transformation
\[
B^{{\rm fat}}\pr{\ast,\cal I,Q\circ-}\to\colim_{\cal I}Q\circ-\to\colim_{\cal I}
\]
exhibits $B^{{\rm fat}}\pr{\ast,\cal I,Q\circ-}$ as an absolute left
derived functor of $\colim_{\cal I}$. If the weak simplicial framing
is excellent, the same conclusion holds for $B$ instead of $B^{{\rm fat}}$.
\item The composite natural transformation
\[
\lim_{\cal I}\to\lim_{\cal I}R\circ-\to C^{{\rm fat}}\pr{\ast,\cal I,R\circ-}
\]
exhibits $C^{{\rm fat}}\pr{\ast,\cal I,R\circ-}$ as an absolute right
derived functor of $\lim_{\cal I}$. If the weak cosimplicial framing
is excellent, the same conclusion holds for $C$ instead of $C^{{\rm fat}}$.
\end{enumerate}
\end{thm}

\begin{rem}
For weak cosimplicial framing arising from enriched categories (Example
\ref{exa:enr->wst}), Vok\v{r}\'{i}nek proved a version of Theorem
\ref{thm:BK} that applies more generally to homotopy weighted colimits
\cite[Theorem 2]{Vok12}. 
\end{rem}

We also prove the following variation of Theorem \ref{thm:BK} for
functor tensor products and functor cotensor products \cite[$\S$4.3]{cathtpy}:
\begin{cor}
\label{cor:BK}Let $\cal C$ be a weakly framed model category with
a cofibrant replacement $Q\to\id_{\cal C}$ and a fibrant replacement
$\id_{\cal C}\to R$. Let $\cal I$ be a small category.
\begin{enumerate}
\item For every projectively cofibrant diagram $W\in\SS^{\cal I^{\op}}$
which is pointwise weakly contractible, the natural transformation
\[
W\otimes_{\cal I}\pr{Q\circ-}\to\colim_{\cal I}
\]
exhibits $W\otimes_{\cal I}\pr{Q\circ-}$ as an absolute left derived
functor of $\colim_{\cal I}$.
\item For every projectively cofibrant diagram $W\in\SS^{\cal I}$ which
is pointwise weakly contractible, the natural transformation
\[
\lim_{\cal I}\to\{W,R\circ-\}^{\cal I}
\]
exhibits $\{N\pr{\cal I_{/-}},R\circ-\}^{\cal I}$ as an absolute
left derived functor of $\lim_{\cal I}$.
\end{enumerate}
\end{cor}

\begin{rem}
In \cite{Hirschhorn}, Hirschhorn \textit{defines} homotopy (co)limits
in framed model categories by the functor tensor products appearing
in Corollary \ref{cor:BK}. In spite of its foundational nature, it
seems that a proof of the equivalence between Hirschhorn's definition
and our definition of homotopy colimits (Definition \ref{def:hocolim})
had not appeared in the literature for some time. Quite recently,
Arkhipov--{\O}rsted finally gave a proof of the equivalence for
\textit{combinatorial} model categories \cite{AO23}. Corollary \ref{cor:BK}
applies to \textit{all} model categories, so it improves on their
result.
\end{rem}

\begin{rem}
After the completion of the paper, the author was made aware of an
independent work of Rodr\'iguez Gonz\'alez, \cite{Rodriguez14},
which considers the Bousfield--Kan formula in the setting of relative
categories. In particular, Corollary \ref{cor:BK} appears as Theorem
4.2 of loc. cit. as a consequence of formalisms of ``simplicial descent
categories.'' While the approach of loc. cit. encompasses more general
examples than ours, our approach is much shorter and elementary. We
should also stress that only a sketch proof is given in loc. cit.,
as explicitly stated by the author.
\end{rem}

\begin{example}
\label{exa:real_hocolim}Let $\cal C$ be a weakly framed model category
with a cofibrant replacement $Q$. Corollary \ref{cor:BK} (applied
to $\cal I=\Del^{\op}_{\inj}$ and $W=\Delta^{\bullet}$) proves that
the composite $\norm -\circ Q:\cal C^{\Del^{\op}_{\inj}}\to\cal C$
is an absolute left derived functor of $\colim_{\Del^{\op}_{\inj}}$.
If further the weak cosimplicial framing is excellent, Corollary \ref{cor:BK}
(combined with Lemma \ref{lem:reedy_end}, Remark \ref{rem:Reedy},
and \cite[Theorems 14.5.4 and 19.3.1]{Hirschhorn}) implies that the
functor $\abs -\circ\overline{Q}:\cal C^{\Del^{\op}}\to\cal C$ is
an absolute left derived functor of $\colim_{\Del^{\op}}$, where
$\overline{Q}$ denotes a Reedy cofibrant replacement functor of $\cal C^{\Del^{\op}}$.
\end{example}

The remainder of this section is devoted to the proof of Theorem \ref{thm:BK}
and Corollary \ref{cor:BK}. Our strategy is to reduce this to the
case of simplicial model categories, using Lemma \ref{lem:simplicial_test}
below. For expositional purposes, we start by recalling a few results
on Reedy categories.
\begin{notation}
Let $\cal I$ and $\cal J$ be Reedy categories \cite[Definition 5.2.1]{Hovey}
with degree functions $\opn{deg}:\opn{ob}\cal I\to\lambda$ and $\deg:\opn{ob}\cal J\to\mu$,
where $\lambda$ and $\mu$ are ordinals. Regard the set $\lambda\times\mu$
as equipped with the lexicographic order. (Thus $\pr{x,y}\leq\pr{x',y'}$
if and only if either $x<x'$, or $x=x'$ and $y\leq y'$.) We will
regard $\cal I\times\cal J$ as a Reedy category by setting $\pr{\cal I\times\cal J}_{+}=\cal I_{+}\times\cal J_{+}$
and $\pr{\cal I\times\cal J}_{-}=\cal I_{-}\times\cal J_{-}$, with
degree function given by $\deg\times\deg:\opn{ob}\cal I\times\opn{ob}\cal J\to\lambda\times\mu$.
\end{notation}

\begin{rem}
\label{rem:Reedy}Let $\cal C$ be a model category, and let $\cal I$
and $\cal J$ be Reedy categories. The latching object of $F\in\cal C^{\cal I\times\cal J}$
at $\pr{i,j}\in\cal I\times\cal J$ fits into the pushout
\[\begin{tikzcd}
	{L_iL_jF} & {L_{i}F(-,j)} \\
	{L_jF(i,-)} & {L_{(i,j)}F.}
	\arrow[from=1-1, to=1-2]
	\arrow[from=1-1, to=2-1]
	\arrow[from=1-2, to=2-2]
	\arrow[from=2-1, to=2-2]
\end{tikzcd}\]This implies the following:
\begin{enumerate}
\item Under the isomorphism of categories $\cal C^{\cal I\times\cal J}\cong\pr{\cal C^{\cal I}}^{\cal J}$,
the model category $\cal C^{\cal I\times\cal J}_{{\rm Reedy}}$ can
be identified with the Reedy model structure on $\pr{\cal C^{\cal J}_{{\rm Reedy}}}^{\cal I}_{{\rm Reedy}}$.
In particular, if $F\in\cal C^{\cal I\times\cal J}$ is Reedy cofibrant,
then for each $i\in\cal I$, the diagram $F\pr{i,-}$ is Reedy cofibrant.
\item Let $\cal M,\cal N,\cal P$ be model categories, let $F\in\cal M^{\cal I}$
and $G\in\cal N^{\cal J}$ be Reedy cofibrant objects, and let $\Phi:\cal M\times\cal N\to\cal P$
be a left Quillen bifunctor. Then the diagram $\Phi\circ\pr{F\times G}:\cal I\times\cal J\to\cal P$
is Reedy cofibrant.
\end{enumerate}
\end{rem}

\begin{lem}
\label{lem:reedy_end}\cite[Theorem 3.3]{AO23} Let $\cal C$ be a
model category, and let $\cal I$ be a Reedy category. The coend functor
\[
\int^{\cal I}:\cal C^{\cal I^{\op}\times\cal I}\to\cal C
\]
is left Quillen with respect to the Reedy model structure.
\end{lem}

We now focus on special coends, namely, (fat) geometric realization.
The following lemma is essentially due to Segal, who proved it in
the context of simplicial topological spaces \cite[Proposition A.1]{Seg74}.
\begin{lem}
\label{lem:Seg}Let $\cal C$ be a model category equipped with a
left Quillen bifunctor $\otimes:\SS\times\cal C\to\cal C$, and let
$X$ be a simplicial object in $\cal C$. If $X$ is Reedy cofibrant,
then the map
\[
\norm X\to\abs X
\]
is a weak equivalence of cofibrant objects.
\end{lem}

\begin{proof}
For each simplicial set $K$, let $\abs K$ and $\norm K$ denote
the geometric realization and the fat geometric realization of $K$,
regarded as a levelwise discrete simplicial object in $\SS$. In other
words, we set
\[
\abs K=\int^{[k]\in\Del}K_{k}\cdot\Delta^{k},\,\norm K=\int^{[k]\in\Del_{\inj}}K_{k}\cdot\Delta^{k}.
\]
Using the co-Yoneda lemma and the Fubini theorem for coends, we obtain
a chain of isomorphisms
\begin{align*}
\norm X & =\int^{[n]\in\Del_{\inj}}\Delta^{n}\otimes X_{n}\\
 & \cong\int^{[n]\in\Del_{\inj}}\Delta^{n}\otimes\int^{[m]\in\Del}\Delta^{m}_{n}\cdot X_{m}\\
 & \cong\int^{[m]\in\Del}\int^{[n]\in\Del_{\inj}}\Delta^{m}_{n}\cdot\Delta^{n}\otimes X_{m}\\
 & \cong\int^{[m]\in\Del}\norm{\Delta^{m}}\otimes X_{m}.
\end{align*}
Similarly, there is an isomorphism $\abs X\cong\int^{[m]\in\Del}\abs{\Delta^{m}}\otimes X_{m}$.
Under these isomorphisms, we can identify $\theta$ with the map
\[
\int^{[m]\in\Del}\norm{\Delta^{m}}\otimes X_{m}\to\int^{[m]\in\Del}\abs{\Delta^{m}}\otimes X_{m}
\]
induced by the natural transformation $\norm -\to\abs -\cong\id_{\SS}$
of functors $\Set^{\Del^{\op}}\to\SS$. Thus, by Remark \ref{rem:Reedy}
and Lemma \ref{lem:reedy_end}, it will suffice to prove the following:
\begin{enumerate}
\item The cosimplicial object $\norm{\Delta^{\bullet}}\in\SS^{\Del}$ is
Reedy cofibrant.
\item For each $m\geq0$, the map $\norm{\Delta^{m}}\to\abs{\Delta^{m}}$
is a weak homotopy equivalence.
\end{enumerate}

For (1), we observe that $\norm{\Delta^{\bullet}}=\iota_{!}\iota^{*}\pr{\Delta^{\bullet}}$,
where $\iota:\Del_{\inj}\to\Del$ denotes the inclusion and $\iota_{!}:\SS^{\Del_{\inj}}\adj\SS^{\Del}:\iota^{*}$
denotes the associated adjunction of the left Kan extension functor
and the restriction functor. Since $\Delta^{\bullet}\in\SS^{\Del}$
is Reedy cofibrant, it will suffice to show that $\iota_{!}$ and
$\iota^{*}$ are left Quillen with respect to the Reedy model structures.
For this, it suffices to show that $\iota^{*}$ is left and right
Quillen. It is clear from the definitions of Reedy model structures
that $\iota^{*}$ is left Quillen. The fact that $\iota^{*}$ is right
Quillen follows from \cite[Lemma 15.3.13]{Hirschhorn}, which says
that Reedy fibrations are pointwise fibrations.

For (2), it suffices to show that the simplicial set $\norm{\Delta^{m}}$
is weakly contractible, because $\abs{\Delta^{m}}\cong\Delta^{m}$
is weakly contractible. For this, define a new category $[m]_{+}$
as follows: Its objects are the integers $0,\dots,m$. There is exactly
one non-identity morphism $f_{ij}:i\to j$ for each pair of integers
$i\leq j$, obeying the composition law $f_{jk}f_{ij}=f_{ik}$. We
can then identify $\norm{\Delta^{m}}$ with the nerve of $[m]_{+}$.
Since $[m]_{+}$ admits a natural transformation to the constant functor
at $m\in[m]_{+}$, its nerve is contractible, as desired.

\end{proof}

We need a few more lemmas before proving Theorem \ref{thm:BK}.
\begin{lem}
\label{lem:simplicial_test}Let $\cal C$ be a model category, and
let $f:X\to Y$ be a morphism of cofibrant objects of $\cal C$. Suppose
that, for every simplicial model category $\cal D$ and every left
Quillen functor $L:\cal C\to\cal D$, the morphism $L\pr f$ is a
weak equivalence. Then $f$ is a weak equivalence.
\end{lem}

\begin{proof}
For each object $C\in\cal C$, we define a functor $\Map^{R}_{\cal C}\pr{-,C}:\cal C\to\SS^{\op}$
by $\Map^{R}_{\cal C}\pr{-,C}=\Hom_{\cal C}\pr{-,C_{\bullet}}$, where
$C\to C_{\bullet}$ is a simplicial resolution of $C$. According
to \cite[Theorem 16.5.4]{Hirschhorn}, the functor $\Map^{R}_{\cal C}\pr{-,C}$
is left Quillen. Since the functors $\{\Map^{R}_{\cal C}\pr{-,C}\}_{C\in\cal C}$
jointly reflects weak equivalences of cofibrant objects \cite[Theorem 17.7.7]{Hirschhorn},
the case where $\cal D=\SS^{\op}$ will in fact suffice, and we are
done.
\end{proof}

\begin{lem}
\label{lem:extradegeneracy}Let $\cal C$ be a weakly cosimplicially
framed model category, and let $X\to X_{-1}$ be an augmented simplicial
object admitting extra degeneracies \cite[$\S$4.5]{cathtpy}. If $X$
is Reedy cofibrant and $X_{-1}$ is cofibrant, then the map
\[
\theta:\norm X\to X_{-1}
\]
is a weak equivalence.
\end{lem}

\begin{proof}
According to Lemma \ref{lem:Seg}, the object $\norm X$ is cofibrant.
Therefore, by Lemma \ref{lem:simplicial_test}, it suffices to show
that, for each simplicial model category $\cal D$ and each left Quillen
functor $L:\cal C\to\cal D$, the morphism $L\theta$ is a weak equivalence.
For this, we observe that the both $L\pr{\Delta^{\bullet}\otimes X}$
and $\Delta^{\bullet}\otimes L\pr X$ are cosimplicial resolutions
of the functor $L\pr X:\Del\to\cal D$. Moreover, by Lemma \ref{lem:reedy_end},
any map $\bf Y\to\bf Y'$ of cosimplicial resolutions of $L\pr X$
induces a weak equiavlence
\[
\int^{\Del_{\inj}}\bf Y\xrightarrow{\simeq}\int^{\Del_{\inj}}\bf Y'.
\]
Since $\opn{csRes}\pr X$ is weakly contractible, we are therefore
reduced to showing that the map
\[
\theta':\int^{[n]\in\Del_{\inj}}\Delta^{n}\otimes L\pr{X_{n}}=\norm{L\pr X}\to L\pr{X_{-1}}
\]
is a weak equivalence. We can factor this map as
\[
\norm{L\pr X}\xrightarrow{\phi}\abs{L\pr X}\xrightarrow{\psi}L\pr{X_{-1}}.
\]
The map $\phi$ is a weak equivalence by Lemma \ref{lem:Seg}, and
the map $\psi$ is a weak equivalence since $X\to X_{-1}$ admits
extra degeneracies \cite[Corollary 4.5.2]{cathtpy}. Hence $\theta'$
is a weak equivalence, as desired.
\end{proof}

For the next two lemmas, the reader should keep in mind the following
situation: During the proof of Theorem \ref{thm:BK}, we will want
to study the fat realization $B^{\mathrm{fat}}\pr{\ast,\cal I,F}\cong\colim_{i\in\cal I}B^{\mathrm{fat}}\pr{y\pr i,\cal I,F}$,
where $y:\cal I\to\Set^{\cal I^{\op}}$ denotes the Yoneda embedding.
By definition, we have
\[
B^{\mathrm{fat}}\pr{\cal I\pr{-,i},\cal I,F}=\int^{[n]\in\Del^{\inj}}\Delta^{n}\otimes\pr{\coprod_{i_{0}\to\cdots\to i_{n}}\cal I\pr{i_{n},i}\cdot F\pr{i_{0}}}.
\]
To analyze the right-hand side, it will be convenient if $\Delta^{n}\otimes-$
commutes with colimits, or at least with coproducts. While this is
not always guaranteed for weak framing, Lemma \ref{lem:tens_set}
ensures this up to weak equivalence. Once we have this, we can use
Lemma \ref{lem:sres_hocolim} to show that $B^{\mathrm{fat}}\pr{\ast,\cal I,F}$
computes the correct colimit in the homotopical setting.
\begin{lem}
\label{lem:tens_set}Let $\cal C$ be a weakly cosimplicially framed
model category, and let $\{C_{s}\}_{s\in S}$ be a collection cofibrant
objects. For every simplicial set $K$, the map
\[
\coprod_{s\in S}\pr{K\otimes C_{s}}\to K\otimes\coprod_{s\in S}C_{s}
\]
is a weak equivalence.
\end{lem}

\begin{proof}
The functors $\coprod_{s\in S}\pr{-\otimes C_{s}}:\SS\to\cal C$ and
$-\otimes\coprod_{s\in S}C_{s}:\SS\to\cal C$ are left Quillen, so
it suffices to verify the claim for $K=\Delta^{0}$. In this case,
the claim is immediate from the definitions of weak cosimplicial framing.
\end{proof}

\begin{lem}
\label{lem:sres_hocolim}Let $\cal D$ be a simplicial model category,
let $F:\cal I\to\cal D$ be a small diagram, and let $\bf F$ be a
cosimplicial resolution of $F$. Define a functor $Q_{\bf F}:\cal I\to\cal D$
by
\[
Q_{\bf F}\pr i=\int^{[n]\in\Del_{\inj}}\coprod_{i_{0}\to\cdots\to i_{n}}\cal I\pr{i_{n},i}\cdot\bf F_{n}\pr{i_{0}}.
\]
Then the map
\[
\theta:\hocolim_{\cal I}Q_{\bf F}\to\colim_{\cal I}Q_{\bf F}
\]
is an isomorphism in $\Ho\pr{\cal D}$.
\end{lem}

\begin{proof}
Using Lemma \ref{lem:reedy_end} and the isomorphism
\[
\colim_{\cal I}Q_{\bf F}\cong\int^{[n]\in\Del_{\inj}}\coprod_{i_{0}\to\cdots\to i_{n}}\bf F_{n}\pr{i_{0}},
\]
we deduce that every morphism $\bf F^{\pr 0}\to\bf F^{\pr 1}$ of
cosimplicial resolutions of $F$ induces a weak equivalence $\colim_{\cal I}Q_{\bf F^{\pr 0}}\xrightarrow{\simeq}\colim_{\cal I}Q_{\bf F^{\pr 1}}$.
Since $\opn{csRes}\pr F$ is a weakly contractible category, it will
therefore suffice to show that the claim holds for \textit{some} simplicial
resolution of $F$.

Replacing $F$ by $\bf F_{0}$, we may assume that $F$ is pointwise
cofibrant. In this case, we can take $\bf F=\Delta^{\bullet}\otimes F$,
and we can identify $\theta$ with the map
\[
\hocolim_{i\in\cal I}B^{{\rm fat}}\pr{\cal I\pr{-,i},\cal I,F}\to\colim_{i\in\cal I}B^{{\rm fat}}\pr{\cal I\pr{-,i},\cal I,F}\cong B^{{\rm fat}}\pr{\ast,\cal I,F}
\]
To show that this is an isomorphism, it suffices (by Lemma \ref{lem:Seg})
to show that the map
\[
\hocolim_{i\in\cal I}B\pr{\cal I\pr{-,i},\cal I,F}\to\colim_{i\in\cal I}B\pr{\cal I\pr{-,i}\cal I,F}\cong B\pr{\ast,\cal I,F}
\]
is an isomorphism. This is classical \cite[Theorem 5.1.1]{cathtpy}.
\end{proof}

We now arrive at the proofs of Theorem \ref{thm:BK} and Corollary
\ref{cor:BK}.
\begin{proof}
[Proof of Theorem \ref{thm:BK}]We will prove (1); part (2) is dual.
The assertion in the second sentence follows from Lemma \ref{lem:Seg},
so we will focus on the assertion in the first sentence. 

For each pointwise cofibrant diagram $F\in\cal C^{\cal I}$ and $W\in\Set^{\cal I^{\op}}$,
we define an object $B^{{\rm fat}}_{{\rm fake}}\pr{W,\cal I,F}\in\cal C$
by 
\[
B^{{\rm fat}}_{{\rm fake}}\pr{W,\cal I,F}=\int^{[n]\in\Del_{\inj}}\coprod_{i_{0}\to\cdots\to i_{n}}W\pr{i_{n}}\cdot\pr{\Delta^{n}\otimes F\pr{i_{0}}}.
\]
Lemma \ref{lem:tens_set} shows that the maps
\[
\left\{ \coprod_{i_{0}\to\cdots\to i_{n}}W\pr{i_{n}}\cdot\pr{\Delta^{m}\otimes F\pr{i_{0}}}\to\Delta^{m}\otimes\pr{\coprod_{i_{0}\to\cdots\to i_{n}}W\pr{i_{n}}\cdot F\pr{i_{0}}}\right\} _{n,m\geq0}
\]
are weak equivalences. Thus, by Lemma \ref{lem:reedy_end}, these
maps induce a weak equivalence $B^{{\rm fat}}_{{\rm fake}}\pr{W,\cal I,F}\xrightarrow{\simeq}B^{{\rm fat}}\pr{W,\cal I,F}$
of cofibrant objects of $\cal C$. Therefore, it suffices to prove
the theorem in the case where $B^{{\rm fat}}$ is replaced by $B^{{\rm fat}}_{{\rm fake}}$.

Let $y:\cal I\to\Set^{\cal I^{\op}}$ denote the Yoneda embedding.
Using the isomorphism $\colim_{i\in\cal I}B^{{\rm fat}}_{{\rm fake}}\pr{y\pr i,\cal I,F}\cong B^{{\rm fat}}_{{\rm fake}}\pr{\ast,\cal I,F}$,\footnote{Note that the map $\colim_{i\in\cal I}B^{{\rm fat}}\pr{y\pr i,\cal I,F}\to B^{{\rm fat}}\pr{\ast,\cal I,F}$
may \textit{not} be an isomorphism, and this is why we need $B^{{\rm fat}}_{{\rm fake}}$.} we are reduced to showing that the composite natural transformation
\[
B^{{\rm fat}}_{{\rm fake}}\pr{y\pr -,\cal I,Q\circ-}\To Q\circ-\To\id_{\cal C^{\cal I}}
\]
is a left deformation for $\colim_{\cal I}$. For this, it suffices
to prove the following:

\begin{enumerate}[label=(\alph*)]

\item Let $F:\cal I\to\cal C$ be a pointwise cofibrant diagram.
For each $i\in I$, the map
\[
\theta:B^{{\rm fat}}_{{\rm fake}}\pr{y\pr i,\cal I,F}\to F\pr i
\]
is a weak equivalence.

\item Let $F_{0},F_{1}$ be pointwise cofibrant diagrams, and let
$\alpha:B^{{\rm fat}}_{{\rm fake}}\pr{y\pr -,\cal I,F_{0}}\to B^{{\rm fat}}_{{\rm fake}}\pr{y\pr -,\cal I,F_{1}}$
be a weak equivalence in $\cal C^{\cal I}$. The map $\colim_{\cal I}\alpha$
is a weak equivalence.

\end{enumerate}

For (a), we factor $\theta$ as
\[
B^{{\rm fat}}_{{\rm fake}}\pr{y\pr i,\cal I,F}\xrightarrow{\phi}B^{{\rm fat}}\pr{y\pr i,\cal I,F}\xrightarrow{\psi}F\pr i.
\]
We have already seen that $\phi$ is a weak equivalence, so it suffices
to show that $\psi$ is a weak equivalence. This follows from Lemma
\ref{lem:extradegeneracy}, because the augmented simplicial object
$B_{\bullet}\pr{y\pr i,\cal I,F}\to F\pr i$ admits extra degeneracies
and the simplicial object $B_{\bullet}\pr{y\pr i,\cal I,F}$ is Reedy
cofibrant. 

Next, we prove part (b). By Lemma \ref{lem:simplicial_test}, it suffices
to show that for each left Quillen functor $L:\cal C\to\cal D$ into
a simplicial model category $\cal D$, the map $L\pr{\colim_{\cal I}\alpha}$
is a weak equivalence. For this, it will suffice to show that, for
each pointwise cofibrant diagram $F\in\cal C^{\cal I}$, the map
\[
\hocolim_{i\in\cal I}L\pr{B^{{\rm fat}}_{{\rm fake}}\pr{y\pr i,\cal I,F}}\to\colim_{i\in\cal I}L\pr{B^{{\rm fat}}_{{\rm fake}}\pr{y\pr i,\cal I,F}}
\]
is an isomorphism in $\Ho\pr{\cal D}$. This follows from the isomorphism
\[
L\pr{B^{{\rm fat}}_{{\rm fake}}\pr{y\pr i,\cal I,F}}\cong\int^{[n]\in\Del_{\inj}}\coprod_{i_{0}\to\cdots\to i_{n}}\cal I\pr{i_{n},i}\cdot L\pr{\Delta^{n}\otimes F\pr{i_{0}}}
\]
and Lemma \ref{lem:sres_hocolim} (applied to the cosimplicial resolution
$\mathbf{F}=L\pr{\Delta^{\bullet}\otimes F}$).
\end{proof}

\begin{proof}
[Proof of Corollary \ref{cor:BK}]We will prove part (1); part (2)
is dual. By \cite[Theorem 19.3.1]{Hirschhorn}, it suffices to prove
the claim for \textit{some }projectively cofibrant diagram $W\in\SS^{\cal I^{\op}}$
which is pointwise weakly contractible. For this, recall from the
proof of Theorem \ref{thm:BK} that the natural transformation
\[
B^{{\rm fat}}_{{\rm fake}}\pr{\ast,\cal I,Q\circ-}\to\colim_{\cal I}Q\circ-\to\colim_{\cal I}
\]
exhibits $B^{{\rm fat}}_{{\rm fake}}\pr{\ast,\cal I,Q\circ-}$ as
an absolute left derived functor of $\colim_{\cal I}$. Now if $F\in\cal C^{\cal I}$
is pointwise cofibrant, we have a chain of isomorphisms
\begin{align*}
B^{{\rm fat}}_{{\rm fake}}\pr{\ast,\cal I,F} & =\int^{[n]\in\Del_{\inj}}\coprod_{i_{0}\to\cdots\to i_{n}}\Delta^{n}\otimes F\pr{i_{0}}\\
 & \cong\int^{[n]\in\Del_{\inj}}\coprod_{i_{0}\to\cdots\to i_{n}}\pr{\int^{i\in\cal I}\cal I\pr{i,i_{0}}\cdot\pr{\Delta^{n}\otimes F\pr i}}\\
 & \cong\int^{i\in\cal I}\pr{\int^{[n]\in\Del_{\inj}}\coprod_{i_{0}\to\cdots\to i_{n}}\pr{\cal I\pr{i,i_{0}}\cdot\Delta^{n}}}\otimes F\pr i\\
 & \cong B^{{\rm fat}}\pr{\ast,\cal I,y\pr -}\otimes_{\cal I}F,
\end{align*}
where $y:\cal I^{\op}\to\SS^{\cal I}$ denotes the Yoneda embedding
(composed with the inclusion $\Set^{\cal I}\hookrightarrow\SS^{\cal I}$).
Thus, to complete the proof, it suffices to show that the diagram
$B^{{\rm fat}}\pr{\ast,\cal I,y\pr -}\in\SS^{\cal I^{\op}}$ is projectively
cofibrant, and that it is pointwise weakly contractible.

By definition, $B^{{\rm fat}}\pr{\ast,\cal I,y\pr -}$ is the fat
geometric realization of the simplicial object $B_{\bullet}\pr{\ast,\cal I,y\pr -}\in\pr{\SS^{\cal I^{\op}}}^{\Del^{\op}}$.
Since representable presheaves are projectively cofibrant, this simplicial
object is Reedy cofibrant with respect to the projective model structure.
Lemma \ref{lem:reedy_end} then shows that $B^{{\rm fat}}\pr{\ast,\cal I,y\pr -}$
is projectively cofibrant. Moreover, for each $i\in\cal I$, Lemma
\ref{lem:Seg} gives us a weak equivalence
\[
B^{{\rm fat}}\pr{\ast,\cal I,y\pr i}\xrightarrow{\simeq}B\pr{\ast,\cal I,y\pr i}.
\]
The right-hand side can be identified with the nerve of $\cal I_{i/}$,
which has an initial object. Hence $B^{{\rm fat}}\pr{\ast,\cal I,y\pr i}$
is weakly contractible, as desired.
\end{proof}

\section{\label{sec:geom}Geometric Realizations of Simplicial Chain Complexes}

In ordinary category theory, colimits are built up from coproducts
and coequalizers. In homotopical settings, coequalizers are replaced
by homotopy colimits over $\Del^{\op}$, often called geometric realizations.
In this sense, geometric realizations are among the most fundamental
types of colimits in homotopical contexts. The goal of this section
is to study this fundamental notion in the category of chain complexes. 

In Subsection \ref{subsec:real_Tot}, we show that the geometric realization
of a simplicial chain complex can be rewritten as the totalization
of the associated double complexes (Proposition \ref{prop:realizatoin_totalization}).
In Subsection \ref{subsec:twomorelemmas}, we record a few more properties
of geometric realization that we will use in Section \ref{sec:main}.

\subsection{\label{subsec:real_Tot}Realizations and Totalizations of Double
Complexes}

In this subsection, we study the relation between the geometric realizations
of simplicial chain complexes and the totalizations of double complexes.
To state the main result of this subsection, we must introduce a bit
of notation.
\begin{defn}
Let $X$ be a semisimplicial object in an additive category $\cal C$.
Its \textbf{Moore complex }is the chain complex $M_{\ast}\pr X$ defined
by
\[
M_{n}\pr X=\begin{cases}
X_{n} & \text{if }n\geq0,\\
0 & \text{if }n<0.
\end{cases}
\]
If $n\geq0$, the differential $M_{n+1}\pr X\to M_{n}\pr X$ is given
by the alternating sum $\sum^{n+1}_{i=0}\pr{-1}^{i}d_{i}$ of face
maps. 

Dually, if $X$ is a semi-cosimplicial object of $\cal C$, its \textbf{Moore
complex} $M_{\ast}\pr X$ is the chain complex defined by 
\[
M_{n}\pr X=\begin{cases}
X_{-n} & \text{if }n\leq0,\\
0 & \text{if }n>0,
\end{cases}
\]
with differential given by the alternating sum of coface maps. 
\end{defn}

\begin{defn}
\label{def:normalized}Let $\cal A$ be an abelian category. Given
a simplicial object $X$ in $\cal A$, its \textbf{normalized semisimplicial
object} $X^{{\rm norm}}$ is defined by
\[
X^{{\rm norm}}_{n}=\begin{cases}
\Ker\pr{\pr{d_{i}}^{n}_{i=1}:X_{n}\to\bigoplus^{n}_{i=1}X_{n-1}} & \text{if }n>0,\\
X_{0} & \text{if }n=0,
\end{cases}
\]
with face maps induced by that of $X$. (Thus all but the $0$th face
maps vanish.) We define the \textbf{normalized chain complex} of $X$
as the Moore complex of $X^{{\rm norm}}$, and denote it by $N_{\ast}\pr X$.
We also let $D_{\ast}\pr X$ denote the subcomplex of $M_{\ast}\pr X$
defined by
\[
D_{n}\pr X=\begin{cases}
\Im\pr{\pr{s_{i}}^{n-1}_{i=0}:\bigoplus^{n-1}_{i=0}X_{n-1}\to X_{n}} & \text{if }n>0,\\
0 & \text{otherwise}.
\end{cases}
\]
We will write $\overline{M}_{\ast}\pr X=M_{\ast}\pr X/D_{\ast}\pr X$.

Dually, if $X$ is a cosimplicial object in $\cal A$, we define its
\textbf{normalized semi-cosimplicial object} $X^{{\rm norm}}$ by
\[
X^{{\rm norm}}_{n}=\Coker\pr{\bigoplus^{n-1}_{i=0}X_{n-1}\to X_{n}}.
\]
The Moore complex of $X^{{\rm norm}}$ is called the \textbf{normalized
chain complex} of $X$ and is denoted by $N_{\ast}\pr X$. We define
a quotient complex $D_{\ast}\pr X$ of $M_{\ast}\pr X$ by setting
$D_{n}\pr X=\Coim\pr{\pr{\sigma_{i}}^{n-1}_{i=0}:X_{n}\to\bigoplus^{n-1}_{i=0}X_{n-1}}$
if $n>0$ and $D_{n}\pr X=0$ otherwise, and set $\overline{M}_{\ast}\pr X=\Ker\pr{M_{\ast}\pr X\to D_{\ast}\pr X}$.

Let $\cal C$ be a preadditive category. A \textbf{double complex}
in $\cal C$ is a chain complex 
\[
\cdots\to X_{n,\ast}\to X_{n-1,\ast}\to\cdots
\]
in the category $\Ch\pr{\cal C}$ of chain complexes. More plainly,
a double complex consists of a collection $\{X_{i,j}\}_{i,j\in\bb Z}$
of objects of $\cal A$ and maps $\partial^{h}_{i,j}:X_{i,j}\to X_{i-1,j}$
and $\partial^{v}_{i,j}:X_{i,j}\to X_{i,j-1}$ which satisfy the relations
\[
\partial^{h}_{i-1,j}\partial^{h}_{i,j}=0,\,\partial^{v}_{i,j-1}\partial^{v}_{i,j}=0,\,\partial^{v}_{i-1,j}\partial^{h}_{i,j}=\partial^{h}_{i,j-1}\partial^{v}_{i,j}
\]
for all $i,j\in\bb Z$. 

Let $X$ be a double complex in $\cal C$. We define the \textbf{direct
sum totalization} $\Tot^{\oplus}\pr X$ of $X$ to be the chain complex
defined by
\[
\Tot^{\oplus}\pr X_{n}=\bigoplus_{k+l=n}X_{k,l},
\]
provided that the direct sum exists. The differential is induced by
the maps $X_{k,l}\xrightarrow{\pr{\partial^{h}_{k,l},\pr{-1}^{k}\partial^{v}_{k,l-1}}}X_{k-1,l}\oplus X_{k,l-1}.$
Dually, we define the \textbf{direct product totalization} $\Tot^{\Pi}\pr X$
to be the chain complex whose $n$th term is the product
\[
\Tot^{\Pi}\pr X_{n}=\prod_{k+l=n}X_{k,l},
\]
and whose differential is induced by the maps $\pr{\partial^{h}_{k+1,l},\pr{-1}^{k}\partial^{v}_{k,l+1}}:X_{k+1,l}\oplus X_{k,l+1}\to X_{k,l}$.
\end{defn}

\begin{notation}
Let $\cal C$ be a preadditive category, let $C_{\ast}$ be a chain
complex in $\cal C$, and let $K$ be a simplicial set. We will write
$K\otimes C_{\ast}=N_{\ast}\pr K\otimes C_{\ast}$, where:
\begin{enumerate}
\item $N_{\ast}\pr K$ denotes the normalized chain complex of the free
simplicial abelian group generated by $K$;
\item we regard $\Ch\pr{\cal C}$ as enriched over the symmetric monoidal
category $\Ch\pr{\bb Z}$ of chain complexes of abelian groups and
tensor products of chain complexes, as explained in \cite[\href{https://kerodon.net/tag/00NN}{Tag 00NN}]{kerodon};
and
\item $N_{\ast}\pr K\otimes C_{\ast}$ denotes the tensor of $N_{\ast}\pr K$
with $C_{\ast}$ with respect to the enrichment in (2). 
\end{enumerate}
The object $K\otimes C_{\ast}$ may not exist in general, but it often
does. For example, if $K$ is \textit{finite}, i.e., has only finitely
many nondegenerate simplices, and if $\cal C$ is additive, then $K\otimes C_{\ast}$
exists and is given by the same formula as in $\Ch\pr{\bb Z}$. We
also let $C^{K}_{\ast}=C^{N_{\ast}\pr K}_{\ast}$ denote the cotensor
of $C_{\ast}$ by $N_{\ast}\pr K$, provided that it exists.
\end{notation}

We can now state the main result of this section.
\begin{prop}
\label{prop:realizatoin_totalization}Let $\cal A$ be an abelian
category.
\begin{enumerate}
\item Suppose that $\cal A$ has countable coproducts.
\begin{enumerate}
\item Let $X$ be a semisimplicial object in $\Ch\pr{\cal A}$. There is
an isomorphism of chain complexes
\[
\Tot^{\oplus}\pr{M_{\ast}\pr X}\cong\norm X
\]
which is natural in $X$.
\item Let $X$ be a simplicial object in $\Ch\pr{\cal A}$. There are isomorphisms
of chain complexes
\[
\Tot^{\oplus}\pr{N_{\ast}\pr X}\cong\Tot^{\oplus}\pr{\overline{M}_{\ast}\pr X}\cong\abs X
\]
which is natural in $X$.
\item Let $X$ be a simplicial object in $\Ch\pr{\cal A}$. The map $\norm X\to\abs X$
is a chain homotopy equivalence. 
\end{enumerate}
\item Suppose that $\cal A$ has countable products.
\begin{enumerate}
\item Let $X$ be a semi-cosimplicial object in $\Ch\pr{\cal A}$. There
is an isomorphism of chain complexes
\[
\Tot^{\Pi}\pr{M_{\ast}\pr X}\cong\Tot^{{\rm fat}}\pr X
\]
which is natural in $X$.
\item Let $X$ be a cosimplicial object in $\Ch\pr{\cal A}$. There are
isomorphisms of chain complexes
\[
\Tot^{\Pi}\pr{N_{\ast}\pr X}\cong\Tot^{\Pi}\pr{\overline{M}_{\ast}\pr X}\cong\Tot\pr X
\]
natural in $X$.
\item Let $X$ be a cosimplicial object in $\Ch\pr{\cal A}$. The map $\Tot\pr X\to\Tot^{{\rm fat}}\pr X$
is a chain homotopy equivalence.
\end{enumerate}
\end{enumerate}
\end{prop}

\begin{rem}
\label{rmk:findim}The isomorphisms of Proposition \ref{prop:realizatoin_totalization}
are available whenever the coproduct or product appearing in the definition
of direct sum or direct product totalization exist. (This follows
from the proof of the proposition.) 

For example, let $\cal A$ be an abelian category. Then isomorphisms
of part (1) of Proposition \ref{prop:realizatoin_totalization} is
available for semi- and non-semi-simplicial objects in $\Ch_{\geq0}\pr{\cal A}$. 

As another example, let us say that a semi-simplicial object $X$
in $\cal A$ is \textbf{finite-dimensional} if there are only finitely
many integers $n$ such that $X_{n}\neq0$. Let us also say that a
simplicial object is \textbf{finite-dimensional} if its normalized
semi-simplicial object is finite-dimensional. For finite-dimensional
semi-simplicial objects and simplicial objects, part (1) of Proposition
\ref{prop:realizatoin_totalization} holds without assuming $\cal A$
has countable coproducts, with the same proof. Likewise, part (2)
of Proposition \ref{prop:realizatoin_totalization} holds without
assuming $\cal A$ has countable products if we restrict our attention
to finite-dimensional cosimplicial and semi-cosimplicial objects (i.e.,
cosimplicial and semi-cosimplicial objects that are finite-dimensional
in the opposite category).
\end{rem}

\begin{warning}
Let $\cal A$ be an abelian category with countable coproducts, and
let $X$ be a simplicial object in $\Ch\pr{\cal A}$. In general,
the isomorphisms provided by Proposition \ref{prop:realizatoin_totalization}
do \textit{not} make the diagram
\[\begin{tikzcd}
	{\operatorname{Tot}^\oplus(M_\ast(X))} & {\Vert X\Vert } \\
	{\operatorname{Tot}^\oplus(\overline{M}_\ast(X))} & {|X|}
	\arrow["\cong", from=1-1, to=1-2]
	\arrow[from=1-1, to=2-1]
	\arrow[from=1-2, to=2-2]
	\arrow["\cong"', from=2-1, to=2-2]
\end{tikzcd}\]commutative. Indeed, the proof of Proposition \ref{prop:realizatoin_totalization}
shows that composite
\[
\Delta^{1}\otimes D_{1}\pr X\to\norm X\to\abs X
\]
is non-zero as long as $D_{1}\pr X$ is non-zero, but the composite
\[
\Delta^{1}\otimes D_{1}\pr X\to\norm X\cong\Tot^{\oplus}\pr{M_{\ast}\pr X}\to\Tot^{\oplus}\pr{\overline{M}_{\ast}\pr X}
\]
is zero.
\end{warning}

The remainder of this subsection is devoted to the proof of Propositions
\ref{prop:realizatoin_totalization}. We start by recalling the following
classical result:
\begin{prop}
\label{prop:1.2.3.17}\cite[Proposition 1.2.3.17]{HA}, \cite[III, Theorem 2.4]{GJ99}
Let $X$ be a simplicial object in an abelian category. The composite
\[
N_{\ast}\pr X\xrightarrow{\phi}M_{\ast}\pr X\xrightarrow{\psi}\overline{M}_{\ast}\pr X
\]
is an isomorphism, and the maps $\phi$ and $\psi$ are chain homotopy
equivalences.
\end{prop}

We need two more lemmas.
\begin{lem}
\label{lem:tot_htpy}Let $\cal C$ be an additive category with countable
coproducts. The functor
\[
\Tot^{\oplus}:\Ch\pr{\Ch\pr{\cal C}}\to\Ch\pr{\cal C}
\]
preserves chain homotopy equivalences. 
\end{lem}

\begin{proof}
Observe that if $G_{\ast}\in\Ch\pr{\bb Z}$ is a chain complex of
free abelian groups of countable ranks and $X\in\Ch\pr{\Ch\pr{\cal A}}$
is a double complex, there is a natural (in $G$ and $X$) isomorphism
\[
G_{\ast}\otimes\Tot^{\oplus}\pr{X_{\ast,\ast}}\cong\Tot^{\oplus}\pr{G_{\ast}\otimes X_{\ast,\ast}}.
\]
Specializing to the case where $G_{\ast}=N_{\ast}\pr{\Delta^{1}}$,
we deduce that $\Tot^{\oplus}$ preserves chain homotopy equivalences. 
\end{proof}

\begin{notation}
Let $\cal C$ be an additive category, and let $C_{\ast}$ be a chain
complex in $\cal C$. For each integer $n$, we let $\sk_{n}\pr{C_{\ast}}$
denote the \textbf{$n$-skeleton} of $C_{\ast}$, which is the subcomplex
of $C_{\ast}$ defined by
\[
\sk_{n}\pr{C_{\ast}}_{k}=\begin{cases}
0 & \text{if }k>n,\\
C_{k} & \text{otherwise.}
\end{cases}
\]
\end{notation}

\begin{lem}
\label{lem:Tot}Let $\cal C$ be an additive category. Let $X\in\Ch\pr{\cal C}^{\Del^{\op}_{\inj}}$
be a semi-simplicial object in $\Ch\pr{\cal C}$. For each $n\geq0$,
the square 
\[\begin{tikzcd}
	{\partial\Delta^n\otimes X_{n,\ast}} & {\operatorname{Tot}^\oplus(\operatorname{sk}_{n-1}(M_\ast(X)))} \\
	{\Delta^n\otimes X_{n,\ast}} & {\operatorname{Tot}^\oplus(\operatorname{sk}_{n}(M_\ast(X)))}
	\arrow[from=1-1, to=1-2]
	\arrow[from=1-1, to=2-1]
	\arrow[from=1-2, to=2-2]
	\arrow[from=2-1, to=2-2]
\end{tikzcd}\]is a pushout in $\Ch\pr{\cal C}$, where the bottom horizontal map
is induced by the maps
\[
\bigoplus_{\sigma\in\Delta^{n}_{k}\,\mathrm{nondegenerate}}X_{nl}\xrightarrow{\pr{\sigma^{*}}_{\sigma}}X_{kl}.
\]
\end{lem}

\begin{proof}
This follows by inspection.
\end{proof}

We now arrive at the proof of Proposition \ref{prop:realizatoin_totalization}.
\begin{proof}
[Proof of Proposition \ref{prop:realizatoin_totalization}]We will
prove assertion (1); assertion (2) is dual. 

For part (a), for each $n\geq0$, set $\norm X^{n}=\int^{[k]\in\Del_{\inj,\leq n}}\Delta^{k}\otimes X_{k}$.
Applying Lemma \ref{lem:Tot} iteratively, we obtain an isomorphism
\[
\Tot^{\oplus}\pr{\sk_{n}\pr{M_{\ast}\pr X}}\cong\norm X^{n}.
\]
By taking the colimit as $n$ tends to $\infty$, we obtain the desired
isomorphism
\[
\Tot^{\oplus}\pr{M_{\ast}\pr X}\cong\norm X.
\]

Next, for part (b), let $X$ be a simplicial object in $\Ch\pr{\cal A}$.
We claim that the composite
\[
\theta:\norm{X^{{\rm norm}}}\to\norm X\to\abs X
\]
is an isomorphism. Combining this with part (a) and Proposition \ref{prop:1.2.3.17},
we obtain the desired isomorphism $\Tot^{\oplus}\pr{M_{\ast}\pr X}\cong\Tot^{\oplus}\pr{N_{\ast}\pr X}\cong\norm{X^{{\rm norm}}}\cong\abs X$. 

For each $n\geq0$, set $\abs X^{n}=\int^{[k]\in\Del_{\leq n}}\Delta^{k}\otimes X_{k}$.
To show that $\theta$ is an isomorphism, it suffices to show that
the map $\theta_{n}:\norm{X^{{\rm norm}}}^{n}\to\abs X^{n}$ is an
isomorphism for all $n$. We prove this by induction on $n$. If $n=0$,
the claim is trivial because $X^{{\rm norm}}_{0}=X_{0}=\overline{M}_{0}\pr X$.
For the inductive step, suppose that $\theta_{n-1}$ is an isomorphism.
We must show that $\theta_{n}$ is an isomorphism. For this, consider
the diagram
\[\begin{tikzcd}
	& {(\partial\Delta^n\otimes X_n)\amalg_{\partial\Delta^n\otimes D_n(X)}(\Delta^n\otimes D_n(X))} & {|X|^{n-1}} \\
	& {\Delta^n\otimes X_n} & {|X|^n}. \\
	{\partial\Delta^n\otimes X^{\mathrm{norm}}_n} & {\Vert X^{\mathrm{norm}}\Vert^{n-1}} \\
	{\Delta^n\otimes X^{\mathrm{norm}}_n} & {\Vert X^{\mathrm{norm}}\Vert^{n}}
	\arrow[from=1-2, to=1-3]
	\arrow[from=1-2, to=2-2]
	\arrow[from=1-3, to=2-3]
	\arrow[from=2-2, to=2-3]
	\arrow[from=3-1, to=1-2]
	\arrow[from=3-1, to=3-2]
	\arrow[from=3-1, to=4-1]
	\arrow[from=3-2, to=1-3]
	\arrow[from=3-2, to=4-2]
	\arrow[from=4-1, to=2-2]
	\arrow[from=4-1, to=4-2]
	\arrow[from=4-2, to=2-3]
\end{tikzcd}\]The front and the back faces are pushouts. The left-hand face is also
a pushout because $X_{n}$ is a direct sum of $X^{{\rm norm}}_{n}$
and $D_{n}\pr X$ by Proposition \ref{prop:1.2.3.17}. Hence the right-hand
face is also a pushout. It then follows from the induction hypothesis
that $\theta_{n}$ is an isomorphism, as desired.

For part (c), let again $X$ be a simplicial object in $\Ch\pr{\cal A}$.
We wish to show that the map $\phi:\norm X\to\abs X$ is a chain homotopy
equivalence. We have just seen that the composite $\norm{X^{{\rm norm}}}\xrightarrow{\phi'}\norm X\xrightarrow{\phi}\abs X$
is an isomorphism, so it suffices to show that the map $\phi'$ is
a chain homotopy equivalence. By part (1), we can identify this map
with $\Tot^{\oplus}\pr{N_{\ast}\pr X}\to\Tot^{\oplus}\pr{M_{\ast}\pr X}$,
which is a chain homotopy equivalence by Proposition \ref{prop:1.2.3.17}
and Lemma \ref{lem:tot_htpy}. The proof is now complete.
\end{proof}

\subsection{\label{subsec:twomorelemmas}Two More Lemmas}

In this subsection, we prove two more lemmas on geometric realization
and totalization, which we use in Subsection \ref{subsec:proof}.

Here are the lemmas we wish to prove:
\begin{lem}
\label{lem:extra_deg_chain}Let $\cal A$ be an abelian category,
and let $X\to X_{-1}$ be an augmented simplicial object $X\to X_{-1}$
in $\Ch\pr{\cal A}$ admitting extra degeneracies. If either $\cal A$
has countable coproducts or $X$ is finite-dimensional, the map
\[
\abs X\to X_{-1}
\]
is a chain homotopy equivalence.
\end{lem}

\begin{lem}
\label{lem:Tot_column_wise_qis}Let $\cal A$ be an abelian category.
\begin{enumerate}
\item Let $f:X\to Y$ be a morphism in $\Ch\pr{\cal A}^{\Del^{\op}}$ such
that, for each $n\geq0$, the map $f_{n}:X_{n,\ast}\to Y_{n,\ast}$
is a quasi-isomorphism. If either $\cal A$ is $\ABF_{\Omega}$ or
$X$ and $Y$ are finite-dimensional, the map $\abs f:\abs X\to\abs Y$
is a quasi-isomorphism.
\item Let $f:X\to Y$ be a morphism in $\Ch\pr{\cal A}^{\Del}$ such that,
for each $n\geq0$, the map $f_{n}:X_{n,\ast}\to Y_{n,\ast}$ is a
quasi-isomorphism. If either $\cal A$ is $\ABF^{*}_{\Omega}$ or
$X$ and $Y$ are finite-dimensional, the map $\Tot\pr f:\Tot\pr X\to\Tot\pr Y$
is a quasi-isomorphism.
\end{enumerate}
\end{lem}

\begin{rem}
Let $\cal A$ be an $\ABF_{\Omega}$ abelian category. Lemma \ref{lem:Tot_column_wise_qis}
implies that the functor $\abs -:\Ch\pr{\cal A}^{\Del^{\op}}\to\Ch\pr{\cal A}$
descends to a homotopy colimit functor in the sense of Definition
\ref{def:hocolim}. Indeed, it is a relative functor, and its right
adjoint $\Sing:\Ch\pr{\cal A}\to\Ch\pr{\cal A}^{\Del^{\op}}$ of $\abs -$,
given by $C\mapsto C^{\Delta^{\bullet}}$, is also relative by Lemma
\ref{lem:kunneth} below. Proposition \ref{prop:realizatoin_totalization}
then gives us an adjunction $\Ho\pr{\abs -}:\Ho\pr{\Ch\pr{\cal A}^{\Del^{\op}}}\adj\Ho\pr{\Ch\pr{\cal A}}:\Ho\pr{\Sing}$.
But $\Ho\pr{\Sing}$ is naturally isomorphic to $\Ho\pr{\delta}$,
so $\Ho\pr{\abs -}$ is a homotopy colimit functor. A similar remark
applies to fat geometric realization; details are left to the readers.
\end{rem}

We can give a proof of Lemma \ref{lem:extra_deg_chain} right away.
\begin{proof}
[Proof of Lemma \ref{lem:extra_deg_chain}]Let $\delta\pr{X_{-1}}\in\Ch\pr{\cal A}^{\Del^{\op}}$
denote the constant simplicial object at $X_{-1}$, and let $p:X\to\delta\pr{X_{-1}}$
denote the map of simplicial object induced by the augmentation. We
wish to show that the map $\abs p$ is a chain homotopy equivalence.
According to Proposition \ref{prop:realizatoin_totalization} and
Remark \ref{rmk:findim}, we can identify this map with $\Tot^{\oplus}\pr{\overline{M}_{\ast}\pr p}$.
Therefore, by Lemma \ref{lem:tot_htpy}, it suffices to show that
the map $\overline{M}_{\ast}\pr p$ is a chain homotopy equivalence
(of chain complexes in $\Ch\pr{\cal A}$). Since simplicial homotopies
of simplicial chain complexes give rise to chain homotopies of Moore
complexes \cite[III, Lemma 2.15]{GJ99}, this follows from Proposition
\ref{prop:1.2.3.17}.
\end{proof}

The proof of Lemma \ref{lem:Tot_column_wise_qis} requires a few preliminaries.
\begin{defn}
\label{def:lim1}Let $\cal A$ be an abelian category with countable
products. Given a tower
\[
\cdots\xrightarrow{f_{2}}A_{2}\xrightarrow{f_{1}}A_{1}\xrightarrow{f_{0}}A_{0}
\]
of objects in $\cal A$, we define $\lim^{1}_{n}A_{n}\in\cal A$ to
be the cokernel of the map
\[
F:\prod_{n\geq0}A_{n}\to\prod_{n\geq0}A_{n}
\]
defined by the requirement that for each morphism $\pr{a_{n}}_{n\geq0}:X\to\prod_{n\geq0}A_{n}$,
we have $F\circ\pr{a_{n}}_{n\geq0}=\pr{f_{n}\pr{a_{n+1}}-a_{n}}_{n\geq0}$.
We define $\colim^{1}_{n}$ dually in abelian categories with countable
coproducts.
\end{defn}

\begin{example}
\label{exa:splitepis}In the situation of Definition \ref{def:lim1},
suppose that each $f_{n}$ has a section $s_{n}:A_{n}\to A_{n+1}$.
Then $\lim^{1}_{n}A_{n}=0$. In fact, $F$ is a split epimorphism.
To see this, define inductively a map $u_{m}:\prod_{n\ge0}A_{n}\to A_{m}$
by $u_{0}=0$ and $u_{m+1}=s_{m}\circ\pr{u_{m}+\opn{pr}_{m}}$, where
$\opn{pr}_{m}:\prod_{n\geq0}A_{n}\to A_{m}$ denotes the $m$th projection.
Then $\pr{u_{n}}_{n\geq0}:\prod_{n\ge0}A_{n}\to\prod_{n\ge0}A_{n}$
is a splitting of $F$.

If we merely assume that each $f_{n}$ is an epimorphism, we might
not have $\lim^{1}_{n}A_{n}=0$, even if $\cal A$ is $\ABF^{*}$
\cite{Nee02}.
\end{example}

The following lemma appears as \cite[Theorem 3.5.8]{WeibelHA}. (The
$\ABF^{*}_{\Omega}$ property is necessary to ensure that we have
the long exact sequence involving $\lim$ and $\lim^{1}$, described
in \cite[Lemma 3.5.2]{WeibelHA}.)
\begin{lem}
\label{lem:3.5.8}Let $\cal A$ be an $\ABF^{*}_{\Omega}$ abelian
category, and let
\[
\cdots\xrightarrow{p_{1}}C_{1}\xrightarrow{p_{0}}C_{0}
\]
be a sequence of epimorphisms of chain complexes in $\cal A$. If
$\lim^{1}_{n}C_{n}=0$, then there is an exact sequence
\[
0\to\lim^{1}_{n}H_{\ast+1}\pr{C_{n}}\to H_{\ast}\pr{\lim_{n}C_{n}}\to\lim_{n}H_{\ast}\pr{C_{n}}\to0
\]
which is natural in the tower $\{C_{n}\}_{n\geq0}$.
\end{lem}

The following lemma is a generalization of the K\"unneth formula.
The proof is almost identical to that of the ordinary case, but we
record the proof for readers' convenience. 

\begin{lem}
\label{lem:kunneth}Let $\cal A$ be an abelian category, and let
$K$ be a finite simplicial set (i.e., has only finitely many nondegenerate
simplices) whose integral homology groups are all free. Then for every
chain complex $C_{\ast}$ in $\cal A$, there is an isomorphism
\[
\bigoplus_{k+l=n}H_{k}\pr K\otimes H_{l}\pr{C_{\ast}}\cong H_{n}\pr{K\otimes C_{\ast}}
\]
which is natural in $K$ and $C_{\ast}$.
\end{lem}

\begin{proof}
We will write $N_{\ast}=N_{\ast}\pr K$. For each integer $n$, set
$Z_{n}=\Ker\pr{N_{n}\to N_{n-1}}$ and $B_{n}=\Im\pr{N_{n+1}\to N_{n}}$.
We will regard $Z_{\ast}$ and $B_{\ast}$ as chain complexes with
trivial differentials. There is an exact sequence
\[
0\to Z_{\ast}\to N_{\ast}\to B_{\ast-1}\to0
\]
of chain complexes of abelian groups. Since $B_{\ast}$ is free in
each degree, this sequence splits in each degree. It follows that
the sequence
\[
0\to Z_{\ast}\otimes C_{\ast}\to N_{\ast}\otimes C_{\ast}\to B_{\ast-1}\otimes C_{\ast}\to0
\]
is also exact. We thus obtain a long exact sequence
\[
\cdots\to H_{n}\pr{Z_{\ast}\otimes C_{\ast}}\to H_{n}\pr{N_{\ast}\otimes C_{\ast}}\to H_{n-1}\pr{B_{\ast}\otimes C_{\ast}}\to H_{n-1}\pr{Z_{\ast}\otimes C_{\ast}}\to\cdots.
\]
Note that the connecting homomorphism $H_{n-1}\pr{B_{\ast}\otimes C_{\ast}}\to H_{n-1}\pr{Z_{\ast}\otimes C_{\ast}}$
is induced by the inclusion $B_{\ast}\otimes C_{\ast}\to Z_{\ast}\otimes C_{\ast}$.

Now since $H_{\ast}\pr K$ is free in each degree, the exact sequence
\[
0\to B_{\ast}\to Z_{\ast}\to H_{\ast}\pr K\to0
\]
of chain complexes (with trivial differentials) splits. So the sequence
\[
0\to H_{n}\pr{B_{\ast}\otimes C_{\ast}}\to H_{n}\pr{Z_{\ast}\otimes C_{\ast}}\to H_{n}\pr{H_{\ast}\pr K\otimes C_{\ast}}\to0
\]
is also split exact. Combining this with the long exact sequence above,
we obtain the desired isomorphism
\[
H_{n}\pr{K\otimes C_{\ast}}\cong H_{n}\pr{H_{\ast}\pr K\otimes C_{\ast}}=\bigoplus_{k+l=n}H_{k}\pr K\otimes H_{l}\pr{C_{\ast}}.
\]
\end{proof}

\begin{lem}
\label{lem:norm_qis}Let $\cal A$ be an abelian category, and let
$f:X\to Y$ be a map of simplicial objects in $\Ch\pr{\cal A}$. The
following conditions are equivalent:
\begin{enumerate}
\item For each $n\geq0$, the map $f_{n}:X_{n,\ast}\to Y_{n,\ast}$ is a
quasi-isomorphism.
\item For each $n\geq0$, the map $N_{n}\pr f:N_{n}\pr X\to N_{n}\pr Y$
is a quasi-isomorphism.
\end{enumerate}
\end{lem}

\begin{proof}
This follows from the Dold--Kan correspondence, which says that there
is a direct sum decomposition $X_{n}=\bigoplus_{[n]\epi[k]}N_{k}\pr X$,
where the index ranges over the surjective poset maps $[n]\to[k]$.
\end{proof}

We now come to the proof of Lemma \ref{lem:Tot_column_wise_qis}.
\begin{proof}
[Proof of Lemma \ref{lem:Tot_column_wise_qis}]We will prove part
(1); part (2) follows by a dual argument. By Proposition \ref{prop:realizatoin_totalization}
and Remark \ref{rmk:findim}, it suffices to show that the map $\Tot^{\oplus}\pr{N_{\ast}\pr f}$
is a quasi-isomorphism. For this, we prove the following:

\begin{enumerate}[label=(\alph*)]

\item For each $n\geq0$, the map $\Tot^{\oplus}\pr{\opn{sk}_{n}\pr{N_{\ast}\pr X}}\to\Tot^{\oplus}\pr{\sk_{n}\pr{N_{\ast}\pr Y}}$
is a quasi-isomorphism.

\item If $\cal A$ has countable coproducts, we have $\colim^{1}_{n}\Tot^{\oplus}\pr{\opn{sk}_{n}\pr{N_{\ast}\pr X}}=\colim^{1}_{n}\Tot^{\oplus}\pr{\opn{sk}_{n}\pr{N_{\ast}\pr Y}}=0$.

\end{enumerate}

If $X$ and $Y$ are finite-dimensional, then part (a) will prove
the claim. If $\cal A$ satisfies $\ABF_{\Omega}$, then Lemma \ref{lem:3.5.8}
and the five lemma show that $\Tot^{\oplus}\pr{N_{\ast}\pr f}$ is
a quasi-isomorphism, and we will be done.

For part (a), we observe that for each $k\geq0$, the map $N_{k}\pr X\to N_{k}\pr Y$
is a quasi-isomorphism by Lemma \ref{lem:norm_qis}. Thus, the claim
follows by induction, using the five lemma, the pushout square of
\ref{lem:Tot}, and the K\"unneth formula (Lemma \ref{lem:kunneth}).

For part (b), we will show that $\colim^{1}_{n}\Tot^{\oplus}\pr{\opn{sk}_{n}\pr{N_{\ast}\pr X}}=0$.
Replacing $X$ by $Y$ throughout, we obtain $\colim^{1}_{n}\Tot^{\oplus}\pr{\opn{sk}_{n}\pr{N_{\ast}\pr Y}}=0$.
For each integer $d\in\bb Z$, the map $\Tot^{\oplus}\pr{\opn{sk}_{n-1}\pr{N_{\ast}\pr X}}_{d}\to\Tot^{\oplus}\pr{\opn{sk}_{n}\pr{N_{\ast}\pr X}}_{d}$
is an inclusion of a direct summand, so it is a split monomorphism.
It follows from Example \ref{exa:splitepis} that $\colim^{1}_{n}\Tot^{\oplus}\pr{\opn{sk}_{n}\pr{N_{\ast}\pr X}}_{d}=0$,
and we are done.
\end{proof}

\section{\label{sec:main}Main Results}

The goal of this section is twofold: The first goal is to state and
prove our main results, of which there are two (Subsection \ref{subsec:proof}).
One of our main result (Theorem \ref{thm:main2_precise}) does not
use the language of model categories, but the other one (Theorem \ref{thm:main1_precise})
does. Our second goal is to show that many model categories on chain
complexes satisfy the hypothesis of the theorem (Subsection \ref{subsec:compat}).

\subsection{\label{subsec:proof}Proofs of the Main Theorems}

In this subsection, we prove the main theorems of this paper (Theorems
\ref{thm:main1} and \ref{thm:main2}). Since we stated these theorems
somewhat vaguely in the introduction, we start by giving precise statements
of these theorems.
\begin{defn}
\label{def:compatiblewithssets}Let $\cal A$ be a bicomplete abelian
category, and let $\mu$ be a model structure on $\Ch\pr{\cal A}$.
We say that $\mu$ is \textbf{simplicially admissible} if $\otimes:\SS\times\Ch\pr{\cal A}_{\mu}\to\Ch\pr{\cal A}_{\mu}$
is a left Quillen bifunctor.
\end{defn}

\begin{rem}
Let $\cal A$ be a bicomplete abelian category, and let $\mu$ be
a simplicially admissible model structure on $\Ch\pr{\cal A}$. Then
the bifunctor $\otimes$ endows $\Ch\pr{\cal A}$ with an excellent
weak simplicial framing, and the bifunctor $\SS^{\op}\times\Ch\pr{\cal A}_{\mu}\to\Ch\pr{\cal A}_{\mu}$,
$\pr{K,C_{\ast}}\mapsto C^{K}_{\ast}$ endows $\Ch\pr{\cal A}$ with
an excellent weak cosimplicial framing.
\end{rem}

In Subsection \ref{subsec:compat}, we will give a sufficient condition
for a model structure on $\Ch\pr{\cal A}$ to be simplicially admissible.

Here are the precise statements of Theorems \ref{thm:main1} and Theorem
\ref{thm:main2}. 
\begin{thm}
\label{thm:main1_precise}Let $\cal A$ be a bicomplete abelian category.
Suppose $\Ch\pr{\cal A}$ is equipped with a simplicially admissible
model structure (Definition \ref{def:compatiblewithssets}). Then:
\begin{enumerate}
\item Let $Q\to\id_{\Ch\pr{\cal A}}$ be a cofibrant replacement in $\Ch\pr{\cal A}$.
The composite natural transformation
\[
B\pr{\ast,\cal I,Q\circ-}\to B\pr{\ast,\cal I,-}\to\colim_{\cal I}
\]
exhibits $B\pr{\ast,\cal I,Q\circ-}$ as an absolute left derived
functor of $\colim_{\cal I}$. 
\item The functor $B\pr{\ast,\cal I,-}:\Ch\pr{\cal A}^{\cal I}\to\Ch\pr{\cal A}$
carries weak equivalences of pointwise cofibrant diagrams to weak
equivalences of cofibrant objects.
\item Let $\id_{\Ch\pr{\cal A}}\to R$ be a fibrant replacement in $\Ch\pr{\cal A}$.
The composite natural transformation
\[
\lim_{\cal I}\to C\pr{\ast,\cal I,-}\to C\pr{\ast,\cal I,R\circ-}
\]
exhibits $C\pr{\ast,\cal I,R\circ-}$ as an absolute right derived
functor of $\lim_{\cal I}$.
\item The functor $C\pr{\ast,\cal I,-}:\Ch\pr{\cal A}^{\cal I}\to\Ch\pr{\cal A}$
carries weak equivalences of pointwise fibrant diagrams to weak equivalences
of fibrant objects.
\end{enumerate}
\end{thm}

In the statement of the next theorem, we regard $\Ch\pr{\cal A}$
as a relative category by declaring that its weak equivalences are
the quasi-isomorphisms.
\begin{thm}
\label{thm:main2_precise}Let $\cal A$ be an abelian category, and
let $\kappa$ be a regular cardinal. 
\begin{enumerate}
\item Suppose that $\cal A$ has $\kappa$-small colimits. The following
conditions are equivalent:
\begin{enumerate}
\item $\cal A$ is an $\ABF_{\kappa}$-abelian category.
\item For every $\kappa$-small category $\cal I$, the natural transformation
\[
B\pr{\ast,\cal I,-}\to\colim_{\cal I}
\]
exhibits $B\pr{\ast,\cal I,-}$ as an absolute left derived functor
of $\colim_{\cal I}:\Ch\pr{\cal A}^{\cal I}\to\Ch\pr{\cal A}$.
\end{enumerate}
\item Suppose that $\cal A$ has $\kappa$-small limits. The following conditions
are equivalent:
\begin{enumerate}
\item $\cal A$ is an $\ABF^{*}_{\kappa}$-abelian category.
\item For every $\kappa$-small category $\cal I$, the natural transformation
\[
\lim_{\cal I}\to C\pr{\ast,\cal I,-}
\]
exhibits $C\pr{\ast,\cal I,-}$ as an absolute right derived functor
of $\lim_{\cal I}:\Ch\pr{\cal A}^{\cal I}\to\Ch\pr{\cal A}$.
\end{enumerate}
\end{enumerate}
\end{thm}

The following are the proofs of Theorems \ref{thm:main1_precise}
and \ref{thm:main2_precise}.
\begin{proof}
[Proof of Theorem \ref{thm:main1_precise}]Assertion (1) is a consequence
of Theorem \ref{thm:BK}. Assertion (2) is a consequence of Lemma
\ref{lem:reedy_end}. The rest follows by a dual argument.
\end{proof}

\begin{proof}
[Proof of Theorem \ref{thm:main2_precise}]We will prove (1); part
(2) follows by a dual argument. For (b)$\implies$(a), we prove the
contrapositive. Suppose that $\cal A$ is not $\ABF_{\kappa}$. Find
a set $I$ of cardinality less than $\kappa$ and a collection of
monomorphisms $\{f_{i}:A_{i}\to B_{i}\}_{i\in I}$ such that $\bigoplus_{i}f_{i}$
is not monic. For each $i\in I$, let $\phi_{i}$ denote the morphism
in $\Ch\pr{\cal A}$ depicted as
\[\begin{tikzcd}
	\cdots & 0 & {A_i} & {B_i} & 0 & \cdots \\
	\cdots & 0 & 0 & {B_i/A_i} & 0 & \cdots.
	\arrow[from=1-1, to=1-2]
	\arrow[from=1-2, to=1-3]
	\arrow[from=1-2, to=2-2]
	\arrow[from=1-3, to=1-4]
	\arrow[from=1-3, to=2-3]
	\arrow[from=1-4, to=1-5]
	\arrow[from=1-4, to=2-4]
	\arrow[from=1-5, to=1-6]
	\arrow[from=1-5, to=2-5]
	\arrow[from=2-1, to=2-2]
	\arrow[from=2-2, to=2-3]
	\arrow[from=2-3, to=2-4]
	\arrow[from=2-4, to=2-5]
	\arrow[from=2-5, to=2-6]
\end{tikzcd}\]Each $\phi_{i}$ is a quasi-isomorphism, but $H_{0}\pr{\bigoplus_{i\in I}\phi_{i}}$
is not an isomorphism. (The object $A_{i}$ is in the $0$th degree.)
So the functor $B\pr{\ast,I,-}\cong\bigoplus_{i\in I}\phi_{i}:\Ch\pr{\cal A}^{I}\to\Ch\pr{\cal A}$
is not a relative functor. In particular, it cannot be a left derived
functor of $\bigoplus_{i\in I}$.

Next, we prove (a)$\implies$(b). The proof will be very similar to
that of Theorem \ref{thm:BK}. Let $y:\cal I\to\Set^{\cal I^{\op}}$
denote the Yoneda embedding. We define a functor $B\pr{y,\cal I,-}:\Ch\pr{\cal A}^{\cal I}\to\Ch\pr{\cal A}^{\cal I}$
by
\[
B\pr{y,\cal I,F}\pr i=B\pr{y\pr i,\cal I,F}.
\]
We have a natural isomorphism of functors $\colim_{\cal I}B\pr{y,\cal I,-}\cong B\pr{\ast,\cal I,-}:\Ch\pr{\cal A}^{\cal I}\to\Ch\pr{\cal A}$.
Thus, by Theorem \ref{thm:2.2.8}, it suffices to show that the natural
transformation $q:B\pr{y,\cal I,-}\To\id_{\Ch\pr{\cal A}^{\cal I}}$
determines a left deformation of $\colim_{\cal I}$. In other words,
we must check the following:

\begin{enumerate}[label=(\roman*)]

\item For each $F\in\Ch\pr{\cal A}^{\cal I}$, the map $B\pr{y\pr i,\cal I,F}\to F\pr i$
is a quasi-isomorphism.

\item The functor $B\pr{\ast,\cal I,-}$ carries pointwise quasi-isomorphisms
in $\Ch\pr{\cal A}^{\cal I}$ to quasi-isomorphisms. 

\end{enumerate}

Assertion (i) is a consequence of Lemma \ref{lem:extra_deg_chain}.
For (ii), let $F\to G$ be a morphism in $\Ch\pr{\cal A}^{\cal I}$,
and suppose that for each $i\in I$, the map $F\pr i\to G\pr i$ is
a quasi-isomorphism. We wish to show that the map 
\[
B\pr{\ast,\cal I,F}\to B\pr{\ast,\cal I,G}
\]
is a quasi-isomorphism. By Lemma \ref{lem:Tot_column_wise_qis}, it
will suffice to show that, for each $n\geq0$, the map
\[
B_{n}\pr{\ast,\cal I,F}\to B_{n}\pr{\ast,\cal I,G}
\]
is a quasi-isomorphism. This follows from our assumptions that $\cal I$
is $\kappa$-small and $\cal A$ is $\ABF_{\kappa}$. The proof is
now complete.
\end{proof}

\subsection{\label{subsec:compat}Simplicially Admissible Model Structures on
$\protect\Ch\protect\pr{\protect\cal A}$}

The goal of this section is to show that many model structures on
chain complexes are simplicially admissible in the sense of Definition
\ref{def:compatiblewithssets} (Example \ref{exa:sad}). 

We start by recalling the projective model structure on nonnegative
chain complexes of abelian groups.
\begin{prop}
\cite[Theorem 7.2 and Proposition 7.19]{DS95}\label{prop:DwSp} The
category $\Ch_{\geq0}\pr{\bb Z}$ of non-negative chain complexes
admits a model structure, called the \textbf{projective model structure},
which has the following descriptions:
\begin{enumerate}
\item The weak equivalences are the quasi-isomorphisms.
\item The fibrations are the maps that induce epimorphisms in the positive
degrees.
\item The cofibrations are the degreewise monomorphisms whose cokernel is
degreewise free.
\end{enumerate}
Moreover, this model structure is cofibrantly generated, with generating
cofibrations $\{0\to S^{0}\}\cup\{i_{n}:S^{n-1}\to D^{n}\mid n\geq1\}$
and generating trivial cofibrations $\{j_{n}:0\to D^{n}\mid n\geq1\}$.
Here $S^{n-1}$ and $D^{n}$ are defined by
\[
S^{n-1}=\pr{\cdots\to0\to0\to\underset{\text{degree }n-1}{\bb Z}\to0\to0\to\cdots}
\]
and 
\[
D^{n}=\pr{\cdots\to0\to0\to\underset{\text{degree }n}{\bb Z}\xrightarrow{\id}\underset{\text{degree }n-1}{\bb Z}\to0\to0\to\cdots}.
\]
\end{prop}

It is clear from the definitions that the functor $N_{\ast}\pr -:\SS\to\Ch_{\geq0}\pr{\bb Z}_{\proj}$
is left Quillen. This implies the following:
\begin{prop}
\label{prop:tens->sad}Let $\cal A$ be a bicomplete abelian category,
and let $\mu$ be a model structure on $\Ch\pr{\cal A}$. If the tensor
bifunctor
\[
\otimes:\Ch_{\geq0}\pr{\bb Z}_{\proj}\times\Ch\pr{\cal A}_{\mu}\to\Ch\pr{\cal A}_{\mu}
\]
is a left Quillen bifunctor, then $\mu$ is simplicially admissible.
\end{prop}

Proposition \ref{prop:tens->sad} gives two practical criteria for
simplicial admissibility of model structures on chain complexes:
\begin{cor}
\label{cor:sufficient_condition}Let $\cal A$ be a bicomplete abelian
category, and let $\mu$ be a model structure on $\Ch\pr{\cal A}$
satisfying the following conditions:
\begin{enumerate}
\item The weak equivalences are the quasi-isomorphisms.
\item A morphism of $\Ch\pr{\cal A}$ is a cofibration if and only if it
is a monomorphism with cofibrant cokernel.
\item A chain complex $\pr{C_{\ast},\partial_{\ast}}\in\Ch\pr{\cal A}$
is cofibrant if and only if its suspension $\Sigma\pr{C_{\ast},\partial_{\ast}}=\pr{C_{\ast-1},-\partial_{\ast}}$
is cofibrant.
\end{enumerate}
Then $\mu$ is simplicially admissible.
\end{cor}

\begin{proof}
By Propositions \ref{prop:DwSp} and \ref{prop:tens->sad}, it suffices
to show that for each cofibration $f:X_{\ast}\to Y_{\ast}$ in $\Ch\pr{\cal A}$,
the following conditions hold:
\begin{itemize}
\item [(a)]The map $S^{0}\otimes f$ is a cofibration, which is a quasi-isomorphism
if $f$ is a quasi-isomorphism.
\item [(b)]For each $n\geq1$, the map 
\[
i_{n}\square f:\pr{S^{n-1}\otimes Y_{\ast}}\amalg_{S^{n-1}\otimes X_{\ast}}\pr{D^{n}\otimes X_{\ast}}\to D^{n}\otimes Y_{\ast}
\]
is a cofibration, which is a quasi-isomorphism if $f$ is a quasi-isomorphism.
\item [(c)]For each $n\geq0$, the map $D^{n}\otimes f$ is a quasi-isomorphism.
\end{itemize}
Part (a) is obvious, because $S^{0}\otimes f$ can be identified with
$f$. Part (b) follows from conditions (2) and (3), since $i_{n}\square f$
is a degreewise monomorphism and its cokernel is $\Sigma^{n}\pr{\Coker\pr f}$.
Part (c) is clear, because both $D^{n}\otimes X_{\ast}$ and $D^{n}\otimes Y_{\ast}$
are contractible chain complexes. The proof is now complete.
\end{proof}

\begin{cor}
\label{cor:sufficient_condition_2}Let $\cal A$ be a bicomplete abelian
category, and let $\mu$ be a model structure on $\Ch\pr{\cal A}$
whose class of weak equivalences is the class of quasi-isomorphisms.
Suppose that $\mu$ satisfies the following condition:
\begin{itemize}
\item [($\ast$)]There is a set $\cal G$ of objects of $\cal A$ such that
the class of cofibrations of $\Ch\pr{\cal A}$ is the smallest class
of morphisms containing $\{i_{m}\otimes X:S^{m-1}\otimes X\to D^{m}\otimes X\mid m\in\bb Z,\,X\in\cal G\}$
and which is stable under retracts, pushouts, and transfinite compositions.
\end{itemize}
Then $\mu$ is simplicially admissible.
\end{cor}

\begin{proof}
As in the proof of Corollary \ref{cor:sufficient_condition}, it suffices
to show that for each pair of integers $n,m$ and for each $X\in\cal G$,
the map
\[
i_{n}\square\pr{i_{m}\otimes X}:\pr{S^{n-1}\otimes D^{m}\otimes X}\amalg_{S^{n-1}\otimes S^{m-1}\otimes X}\pr{D^{n}\otimes S^{m-1}\otimes X}\to D^{n}\otimes D^{m}\otimes X
\]
is a cofibration in $\mu$. Using the description of cofibrations
in Proposition \ref{prop:DwSp}, we find that a degree shift of the
map $f:\pr{S^{n-1}\otimes D^{m}}\amalg_{S^{n-1}\otimes S^{m-1}}\pr{D^{n}\otimes S^{m-1}}\to D^{n}\otimes D^{m}$
is a cofibration in $\Ch_{\geq0}\pr{\cal A}_{\proj}$. It follows
that $f$ is a retract of a transfinite composition of pushouts of
maps in $\{i_{k}:S^{k-1}\to D^{k}\mid k\in\bb Z\}$. Condition ($\ast$)
then implies that $f\otimes X$ is a cofibration in $\mu$, and we
are done.
\end{proof}

We can now show that many model structures on chain complexes are
simplicially admissible.
\begin{example}
\label{exa:sad}Let $\cal A$ be a bicomplete abelian category.
\begin{enumerate}
\item Suppose that $\Ch\pr{\cal A}$ admits the \textbf{injective model
structure}, i.e., a model structure whose cofibrations are the monomorphisms
and weak equivalences are the quasi-isomorphisms. (This model structure
exists if $\cal A$ is a Grothendieck abelian category \cite[Proposition 1.3.5.3]{HA}.)
Corollary \ref{cor:sufficient_condition} says that the injective
model structure is simplicially admissible.
\item Suppose that $\Ch\pr{\cal A}$ admits the \textbf{projective model
structure}, i.e., a model structure in which fibrations are the epimorphisms
and weak equivalences are the quasi-isomorphisms. (For instance, this
happens if $\cal A$ is the category of left $R$-modules over a ring
\cite[Theorem 2.3.11]{Hovey}.) A dual argument to (1) shows that
the projecitve model structure is simplicially admissible.
\item Let $\cal A$ be a bicomplete abelian category. Model structures on
$\Ch\pr{\cal A}$ constructed by using Hovey and Gillespie's work
\cite[Theorem 7.9]{Hovey_Cotorsion} satisfy the hypotheses of Corollary
\ref{cor:sufficient_condition}, so they are simplicially admissible.
\item Let $\cal A$ be a Grothendieck abelian category. Model structures
on $\Ch\pr{\cal A}$ constructed by using Cisinski and D\'{e}glise's
work \cite[Theorem 2.5]{CD09} satisfy the hypothesis of Corollary
\ref{cor:sufficient_condition_2}, so they are simplicially admissible.
\end{enumerate}
\end{example}

\subsection*{Acknowledgment}

I am grateful to an anonymous referee for substantially improving
this paper in numerous ways. They suggested a link between this work
and framing on model categories, pointed out an error in a proof of
a proposition in an earlier draft, offered a fix for this error, and
suggested simpler proofs of some results. I also thank Timothy Porter
and Birgit Richter for informing me of related work. During the revision,
I was supported by JSPS KAKENHI Grant Number 24KJ1443.

\providecommand{\bysame}{\leavevmode\hbox to3em{\hrulefill}\thinspace}
\providecommand{\MR}{\relax\ifhmode\unskip\space\fi MR }
\providecommand{\MRhref}[2]{%
  \href{http://www.ams.org/mathscinet-getitem?mr=#1}{#2}
}
\providecommand{\href}[2]{#2}

\end{document}